\documentclass[11pt, reqno]{amsart}
\usepackage[margin=1in]{geometry}                % See geometry.pdf to learn the layout options. There are lots.
\usepackage{graphicx}
\usepackage{amssymb}
\usepackage{epstopdf}
\usepackage{tikz-cd}
\usepackage{forest}
\usepackage[colorlinks=true, pdfstartview=FitV, linkcolor=blue, citecolor=blue, urlcolor=blue]{hyperref}
\usepackage{tikz}
\usepackage[listings]{tcolorbox}
\usepackage[all, cmtip]{xy}
\usepackage{caption}
\usepackage{subcaption}
\usepackage{mathrsfs}
\usepackage{paralist}
\usepackage{enumitem}
\usepackage{verbatim}
\usepackage{animate}
\usepackage[smalltableaux]{ytableau}
\usepackage[nameinlink]{cleveref}
\usepackage{accents}
\usepackage{xcolor}
\usepackage{tabularx}
\usepackage{pgfplots}
\usetikzlibrary{3d, perspective}
\crefname{defn}{definition}{definitions}
\Crefname{def}{Definition}{Definitions}
\crefname{thm}{theorem}{theorems}
\Crefname{thm}{Theorem}{Theorems}

\usepackage{color}

\definecolor{lightpink}{RGB}{250, 107, 198}
\definecolor{darkpurple}{RGB}{128, 0, 128}
\definecolor{mypink}{RGB}{250, 127, 209}
\definecolor{myblue}{RGB}{82, 124, 239}
\definecolor{mygreen}{RGB}{150, 240, 130}

\newcommand{\N}{{\mathbb N}}

\newcommand{\Q}{{\mathcal Q}}
\renewcommand{\L}{{\mathcal L}}

\newcommand{\ba}{{\mathbf a}}
\newcommand{\bb}{{\mathbf b}}

\newcommand{\bx}{{\mathbf x}}

\newcommand{\M}{{\mathcal M}}

\renewcommand{\P}{{\mathcal P}}

\newcommand{\cP}{{\mathcal P}}

\def\lex{\operatorname{lex}}
\newcommand{\dSdw}{\Delta}
\newcommand{\uSdw}{\rotatebox[origin=c]{180}{$\Delta$}}
\newcommand{\di}{\diamond}

\def\Seg{\operatorname{Seg}}

\newcommand{\uwidehat}[1]{%
	\mathpalette\douwidehat{#1}%
}
\makeatletter
\newcommand{\douwidehat}[2]{%
	\sbox0{$\m@th#1\widehat{\hphantom{#2}}$}%
	\sbox2{$\m@th#1x$}
	\sbox4{$\m@th#1#2$}
	\dimen0=\ht0
	\advance\dimen0 -.8\ht2
	\dimen2=\dp4
	\rlap{%
		\raisebox{\dimexpr\dimen0-\dimen2}{%
			\scalebox{1}[-1]{\box0}%
		}%
	}%
	{#2}%
}
\makeatother

\theoremstyle{plain} %% This is the default, anyway
\newtheorem{thm}{Theorem}[section]

\newtheorem*{introthm*}{Theorem}

\newtheorem{cor}[thm]{Corollary}
\newtheorem{lem}[thm]{Lemma}
\newtheorem{prop}[thm]{Proposition}

\newtheorem*{oprobl*}{Open Problem}

\newtheorem{introprop}{Proposition}
 
\newtheorem{introthm}[introprop]{Theorem}

\newtheorem{conj}[thm]{Conjecture}

\theoremstyle{definition}
\newtheorem{defn}[thm]{Definition}

\newtheorem{ex}[thm]{Example}

\newtheorem{claim}[thm]{Claim}

\theoremstyle{remark}
\newtheorem{rem}[equation]{Remark}
\newtheorem{notation}[thm]{Notation}

\numberwithin{equation}{section}  %% equation numbering

\DeclareMathOperator{\dia}{\raisebox{-2pt}{\scalebox{1.5}{$\Diamond$}}}

\DeclareMathOperator*{\bigdia}{\raisebox{-4pt}{\scalebox{2}{$\Diamond$}}}

\pgfplotsset{compat=1.17}
\makeatletter
\pgfkeys{/tikz/isometric/.style={x={(-0.866cm, -0.5cm)}, y={(0.866cm, -0.5cm)}, z={(0cm, 1cm)}}}
\makeatother

\title{Constructions of Macaulay Posets and Macaulay Rings}

\author[]{Penelope Beall}
\address{University of California--Davis, 1 Shields Avenue, Davis, CA 95616, United States}
\email{pbeall@ucdavis.edu}

\author[]{Erenay Boyali}
\address{Istanbul Bilgi University, Piyalepasa, Cemiyet, 15 Beyoglu/Istanbul, Turkey}
\email{erenay.boyali@hotmail.com}

\author[]{Nancy Chen}
\address{Cornell University, 616 Thurston Ave, Ithaca, NY 14853, United States}
\email{nc494@cornell.edu}

\author[]{Ellen Chlachidze}
\address{University of Southern California, 700 Childs Way, Los Angeles, CA 90089-0911, United States}
\email{chlachid@usc.edu}

\author[]{Toan Trong Dao}
\address{Faculty of Mathematics and Computer Science, University of Science, Vietnam National University, Ho Chi Minh City, Vietnam}
\email{daotrongtoan.dtt4@gmail.com}

\author[]{Frederic Garvey}
\address{Truman State University, 100 E Normal Ave, Kirksville, MO 63501 United States}
\email{fgarvey125@gmail.com}

\author[]{Mitchell Johnson}
\address{Hamilton College, 198 College Hill Rd, Clinton, NY 13323 United States}
\email{mxjohnso@hamilton.edu}

\author[]{Yu Olivier Li}
\address{University of St Andrews, Mathematical Institute, St Andrews KY16 9SS, United Kingdom}
\email{yol1@st-andrews.ac.uk}

\author[]{Nikola Kuzmanovski}
\address{University of Notre Dame, 138 Hayes-Healy, Notre Dame, IN, 4618, United States}
\email{nkuzmano@nd.edu}

\author[]{Kelvin Ma}
\address{Iowa State University, Ames, IA 50011, United States}
\email{kelvinma78@gmail.com}

\author[]{Treanungkur Mal}
\address{%Theoretical Statistics and Mathematics Unit, 
Indian Statistical Institute, 8th Mile, Mysore Road, RVCE Post, Bangalore 560059, India}
\email{maltreanungkur@gmail.com}

\author[]{Rukshan Marasinghe Mudiyanselage}
\address{%Department of Mathematics, 
Indiana University Bloomington, 107 S Indiana Avenue, Bloomington, IN, 47405, United States}
\email{rmarasin@iu.edu}

\author[]{Quinlan Mayo}
\address{%College of Information and Computer Sciences, 
University of Massachusetts, Amherst, MA, 01002, United States}
\email{qmayo@umass.edu}

\author[]{Nava Minsky-Primus}
\address{Yale University, New Haven, CT 06520, United States}
\email{navaminskyprimus@gmail.com}

\author[]{Alexandra Seceleanu}
\address{University of Nebraska--Lincoln, 203 Avery Hall, Lincoln, NE 68588, United States}
\email{aseceleanu@unl.edu}

\author[]{Sriram Veerapaneni}
\address{Indian Institute of Technology Kanpur, Kalyanpur, Kanpur, Uttar Pradesh 208016, India}
\email{sriramcv22@iitk.ac.in}

\begin{document}



\begin{abstract}
    A poset is Macaulay if its partial order and an additional total order interact well. Analogously, a ring is Macaulay if the partial order defined on its monomials by division interacts nicely with any total monomial order.  We investigate methods of obtaining new structures through combining Macaulay rings and posets by means of certain operations inspired by topology. We  examine whether these new structures retain the Macaulay property, identifying new classes of posets and rings for which the operations preserve the Macaulay property. 
\end{abstract}

\maketitle
\setcounter{tocdepth}{1}
\tableofcontents

\vspace{-2em}
\section{Introduction}

A poset is Macaulay if there exists an additional total order such that initial segments on rank \(n\) have a minimum upper shadow which is itself an initial segment at rank \(n + 1\).  This notion is rigorously formulated in \Cref{def: Macaulay}.
Named after the mathematician Francis S.~Macaulay, these posets generalize a property satisfied by  the monomials in a polynomial ring, partially ordered by divisibility, to a purely combinatorial setting. A Macaulay poset is characterized by a ranked structure and a total order that interacts with the partial order in a favorable way, allowing to establish bounds on the sizes of subsets of a given rank in an order ideal of the poset. 

Macaulay posets serve as a powerful tool in understanding the interplay between algebraic invariants of graded rings, such as Hilbert functions, and their combinatorial structure. Indeed, their seminal application devised by Macaulay in \cite{Mac} was to determine all Hilbert function of homogeneous ideals of the polynomial ring. Nowadays,  applications of Macaulay posets span both commutative algebra  and enumerative combinatorics, making them a rich subject of study in modern mathematics. In extremal combinatorics, Macaulay posets are featured in the study of isoperimetric problems with celebrated results due to Katona \cite{Katona}, Kruskal \cite{Kruskal}, and Clements--Lindstr\"om \cite{CL}.  In commutative algebra there is an emerging theory of Macaulay rings, that is, rings whose monomial poset is a Macaulay poset; see \cite{Mermin, MP, MP2, Kuz}. Surveys on Macaulay posets can be found in \cite{BL} as well as in \cite[Chapter 8]{Engel}. Important examples of Macaulay posets include the poset of monomials in a polynomial ring in any number of variables with the lexicographic order, as shown by Macaulay \cite{Mac}; any product of chains with respect to an appropriate  lexicographic order, as shown by Clements--Lindstr\"om \cite{CL} -- we refer to these as Clements-Lindstr\"om posets throughout -- and any product of isomorphic spider posets with legs of equal lengths, as shown by Bezrukov--Els\"asser \cite{BE}.

In this paper we consider how the Macaulay property interacts with poset operations. Certain instances of this question have been studied before, primarily for  Cartesian products of posets. Iterated Cartesian products of a poset with itself have been studied with respect to the lexicographic total order in \cite{uwsuper}. Special cases of the ring-theoretic analogue construction, which entails taking a tensor product of  two algebras over a common subfield, can be found in \cite[Theorem 4.1]{MP}. See also \cite[Theorem 10]{BL}  for related investigations.

The novel aspect of this paper is considering poset operations that are inspired by  topological  constructions. Specifically, in topology, the fiber product of topological spaces is a construction that combines two spaces over a common base space  in a way that respects certain maps between them and the connected sum of two topological spaces is a new space formed by removing an open neighborhood  from each space and gluing the resulting boundaries together. 
Inspired by these considerations, we introduce and study poset analogues  which we term wedge product (\Cref{def: wedge}), diamond product (\Cref{def: diamond}), and fiber product (\Cref{def: fiber}), respectively.
The fiber product of two posets combines two posets which contain copies of the same sub-poset by taking the union of the posets with the two copies of the subposet identified. The special case of the fiber product where unique minimal elements of two posets are identified is termed wedge product.  When unique minimal elements and unique maximal elements of two posets are respectively identified the cnstruction is termed diamond product. This is a special case of a more general connected sum operation, described in \cite{BKLS}, which we shall not make use of in this paper.

Ring-theoretic constructions which match both these topological and poset operations are available in commutative algebra. The fiber product is the pullback  in the category of rings, defined in \eqref{eq: FP}. The connected sum of rings is an operation defined in \cite{AAM}. When  these operations act upon the cohomology rings of topological spaces $H^\bullet(X)$ and $H^\bullet(Y)$, the fiber product of these rings recovers the cohomology algebra of the fiber product of the topological spaces, that is $H^\bullet(X)\times_K H^\bullet(Y)=H^\bullet(X\times Y)$. Likewise the connected sum of the rings gives the cohomology algebra of the topological connected sum. These analogies provide the names for the ring theoretic constructions.
Our  insight is that if we consider monomial posets $\M_A$, $\M_B$ of rings $A, B$ instead of topological spaces, analogous relationships hold, for example  $\M_{A\times_K B}=\M_A\times_K \M_B$ (see \Cref{prop: FP}).  Since the monomial posets of graded rings are closed under these new poset operations, they are well-suited to study the Macaulay property for  rings. This  feature is absent from other poset operations such as  the disjoint union.

Our main objective is to determine under which conditions the Macaulay property is preserved under the operations we introduce on posets, deducing  consequences on the Macaulay property for some new classes of commutative rings. This is a fruitful endeavor as  few  classes of commutative rings are known to be Macaulay beyond the classical setting of Clements--Lindstr\"om rings; this includes colored square-free rings \cite{MerminMurai}, Chong's rings \cite{Chong}, certain rings of embedding dimension two \cite{Abdelfatah}, and the rings discussed by Kuzmanovski in \cite{Kuz}.

From their respective constructions, it is apparent that the wedge and diamond product are most akin to the disjoint union of posets.  Thus in studying these two operations we first relate them to disjoint unions so that we may utilize results of Clements \cite{Clements} on the Macaulay property for disjoint unions of copies of a Macaulay poset endowed with an additional property called additivity. Our first  result is the following.

\begin{introthm}[\Cref{prop:equiv union wedge diamond}, \Cref{thm: equiv union wedge diamond}]\label{introthmA}
Let $\P_1,\ldots, \P_n$ be posets of the same rank having unique minimal and maximal elements. Then the disjoint union being Macaulay implies the wedge product is Macaulay which implies the diamond product is Macaulay. In symbols,
\[
\bigsqcup_{i=1}^n\P_i \text{ is Macaulay} \Rightarrow \bigvee_{i=1}^n\P_i \text{ is Macaulay} \Rightarrow \bigdia_{i=1}^n\P_i \text{ is Macaulay}.
\]
Moreover if all $\P_i$ are isomomorphic to the same Macaulay poset $\P$ satisfying a mild condition, then any of the three posets above are Macaulay for $n\geq 2$ if and only if $\P$ is additive.
\end{introthm}

\begin{comment}
\begin{introprop}[\Cref{thm: equiv union wedge diamond}]
Let $\P$ be a Macaulay poset. The following  are equivalent:
\begin{enumerate}
\item $\P$ is additive.
\item the disjoint union/wedge product/diamond product of two copies of $\P$ is a Macaulay poset with respect to the union simplicial total order.
\item the iterated disjoint union/wedge product/diamond product of any number (at least two) of copies of $\P$ is  Macaulay  with respect to the union simplicial total order.
\end{enumerate}
\end{introprop}
\end{comment}

Surprisingly, we find that the situation is very different when taking wedge and diamond products of posets that are additive, but not isomorphic, in that the resulting posets are rarely Macaulay. To illustrate this we focus our attention on operations between Clements--Lindstr\"om posets, which  are Macaulay by a celebrated theorem of Clements-Lindstr\"om \cite{CL} and satisfy the additivity property by \cite{Clements1}. Our second series of results characterizes when operations on Clements--Lindstr\"om posets result in Macaulay posets.

\begin{introthm}[\Cref{thm: classification diamond product of boxes},  \Cref{thm: wedge 2 boxes different size}, \Cref{prop: wedge box classification}  ]\label{introthm B}
Let $\P$ and $\mathcal{Q}$ be Clements--Lindstr\"om posets.
\begin{enumerate}
\item The diamond product of $\P$ and $\Q$ is Macaulay if and only if either the two posets are isomorphic or one of them is a path and the other is two-dimensional, with a side of length two.
\item If $\P$ and $\Q$ are two-dimensional of sizes $m\times n$ and $m'\times n'$ , respectively, such that $m \leq n$, $m' \leq n'$, and $m,m' > 1$, then the wedge product of $\P$ and  $\Q$ is Macaulay if and only if $m \leq m'$ and $n \leq n'$ or vice versa.
\end{enumerate}
\end{introthm}

While it is a generalization of wedge product, fiber product proves to be simultaneously more flexible and more challenging to analyze. In this paper we commence a full analysis by studying fiber products of two-dimensional Clements--Lindstr\"om posets. We term the resulting posets  heart-shaped posets (see \Cref{heart} for an illustration). We characterize the Macaulay property in the family of heart-shaped posets. A surprising feature of our result is that different members of this family require different total orders for the Macaulay property to be satisfied. This indicates that additional flexibility with total orders will prove crucial in analyzing more complicated constructions.

\begin{introthm}[\Cref{thm: heart}]\label{introthm C}
The fiber product of $a_0\times a_1$ and $b_0\times b_1$ Clements--Lindstr\"om posets over a $c_0\times c_1$ Clements--Lindstr\"om poset, where $c_i = min\{a_i,b_i\}$, is Macaulay if and only if any of the following hold:
 \begin{enumerate} 
 \item $a_0\leq b_0$ and $a_1\leq b_1$, 
 \item $b_0\leq a_0$ and $b_1\leq a_1$,
   \item $b_0<a_0$ and $a_1<b_1$ and
   
       \begin{inparaenum} 
        \item $a_0=b_1$ or
        \item $b_1<a_0$ and $b_1+b_0 \leq a_0 + a_1$ or
        \item $a_0<b_1$ and $a_0+a_1 \leq b_0 + b_1$,
\end{inparaenum}
    \item condition (3) holds with $a_i$ replaced by $b_i$ and vice versa.
\end{enumerate}
\end{introthm}

Utilizing a correspondence between ring operations and poset operations developed in \Cref{prop: FP}, \Cref{introthm B} and \Cref{introthm C} lead to the following algebraic corollaries:

\begin{introthm}\label{introthm D}
\begin{enumerate}
    \item For integers $a_i,b_j\geq 2$ such that $\sum_{i=1}^n a_i-n=\sum_{i=1}^m b_j-m$ the ring  
     {\small
     \[
    \frac{K[x_1,\ldots, x_n,y_1,\ldots, y_m]}{(x_1^{a_1},\ldots, x_n^{a_n}, y_1^{b_1},\ldots, y_m^{b_m})+( x_iy_j:1\leq i\leq n, 1\leq j\leq m)+(x_1^{a_1-1}\cdots x_n^{a_n-1}-y_1^{b_1-1}\cdots y_m^{b_m-1})}
    \]
    }
is Macaulay if and only if any of the following conditions are satisfied
\begin{enumerate}
\item $n=1$ and $m=2$ and $\min\{b_1,b_2\}=2$  
\item $n=2$ and $m=1$ and $\min\{a_1,a_2\}=2$.
\item $n=m$ and the multisets $\{a_1, \ldots, a_n\}$ and $\{b_1, \ldots, b_m\}$ are equal.
\end{enumerate}
    \item For integers $a_i,b_j\geq 2$ such that $a_1\leq a_2$ and $b_1\leq  b_2$ the ring
    \[
     \frac{K[x_1,x_2,y_1,y_2]}{(x_1^{a_1}, x_2^{a_2}, y_1^{b_1},y_2^{b_2}, x_1y_1,x_1y_2, x_2y_1, x_2y_2)}
    \]
    is Macaulay if and only if either (a) $a_1\leq b_1$ and $a_2\leq b_2$ or (b) $b_1\leq a_1$ and $b_2\leq a_2$.
    \item For integers $a_i,b_j\geq 1$ the ring
    \[
     \frac{K[x,y]}{(x^{a_0}, y^{a_1})\cap(x^{b_0},y^{b_1})}
    \]
    is Macaulay if and only if any of the conditions in \Cref{introthm C} hold.
\end{enumerate}
\end{introthm}

Our paper is organized as follows: in \Cref{s:background} we recall the necessary background on posets and in \Cref{s: poset operations} we introduce the poset and ring operations we are interested in. We study iterated disjoint unions, wedge, and diamond products in \Cref{s: uwd}. In \Cref{s: wd box} we analyze wedge and diamond products of Clements--Lindstr\"om posets and in \Cref{s: fiber box} we study fiber products of Clements--Lindstr\"om posets. We end the paper with \Cref{s: examples} containing conjectures, examples and counterexamples related to taking Cartesian products of Macaulay posets. While this operation has been more extensively studied, it continues to afford open questions for future investigation (cf. \Cref{conjecture2}).

\smallskip 

\paragraph{\bf Acknowledgements.} Computations leading to several results in this paper were
performed using the computer algebra system Macaulay2 \cite{M2}, specifically the package \texttt{MacaulayPosets} developed by some of the authors of this paper \cite{BKLS}. The software tool \cite{java} was also used for computer-assisted verifications. 

We acknowledge the support of NSF DMS--2341670
for the Polymath Jr.\,program, which facilitated our collaboration. 
Seceleanu was partially supported by NSF grant DMS--2401482.

\section{Background}\label{s:background}

In this paper $\mathbb{N}$ denotes the set of nonnegative integers.

\begin{defn}\label{def: partial order}
A {\bf partially ordered set (poset)} is a set $\P$ which has an order relation on it, that is, a binary relation denoted by $\leq$ which satisfies
\begin{itemize}
\item $a\leq a$ for all $a\in \P$
\item if for some $a,b\in \P$ we have $a\leq b$ and $b\leq a$, then $a=b$
\item if for some $a,b,c \in \P$ we have $a\leq b$ and $b\leq c$, then $a\leq c$.
\end{itemize}
As per usual, we write $a<b$ to denote that $a\leq b$ and $a\neq b$.

A {\bf total order}  on a set $\P$ is a binary relation that satisfies the three conditions in \Cref{def: partial order} but for which every pair of elements of $\P$ is comparable.
\end{defn}

A poset $\P$ is sometimes pictured as a graph, called the {\bf Hasse graph} of $\P$. In this graph, an edge connects two elements $a,b\in \P$ so that $b$ {\bf covers} $a$, that is, $a<b$ and there is no $c$ that satisfies $a<c<b$.  In the Hasse graph $a$ will be placed lower than $b$. All  other inequalities can be inferred from this information.

\begin{defn}
A poset $\P$ is {\bf ranked} if there is a function $r : \P \to \N$ called a {\em rank function} for which the following conditions hold:
\begin{itemize}
\item  There is a minimum element $x \in \P$ such that $r(x) = 0$.
\item If  $b$ covers $a$ in $\P$ then $r(b)=r(a)+1$.
\end{itemize}
The rank of $\P$ is the largest rank of an element of $\P$.
For $d\in \N$, the $d$-th {\bf level} of $\P$ is $\P_{[d]} =\{x\in \P \mid r(x)=d\}$.
\end{defn}

The following are examples of ranked posets with level numbers increasing consecutively from 0 for the lowest pictured element. We reverse the order more standardly used for some of these posets to render the correspondence with algebraic posets more transparent.

\begin{ex}
A {\em path poset of length $d$} is a poset on $d+1$ vertices $0,\ldots, d$ which is totally ordered by the ordering of $0,\ldots, d$ as real numbers. The Hasse graph for $d=2$ is
\begin{center}
\includegraphics[height=2cm]{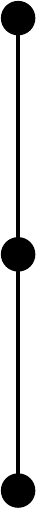}
\end{center}
\end{ex}

Algebraic structures provide an  important source of examples for posets.

\begin{ex}\label{ex: monomial poset}
 Let $R=K[x_1,\dots, x_n]$ be the polynomial ring on variables $x_1,\dots, x_n$ with coefficients in a field $K$. Then the set of all monomials in $R$ equipped with the divisibility relation, that is,  the relation $\bx^\ba =x_1^{a_1}x_2^{a_2}\cdots x_n^{a_n}\mid \bx^\bb=x_1^{b_1}x_2^{b_2}\cdots x_n^{b_n}$ if and only if $\ba\leq \bb$, that is, $a_i\leq b_i$ for $1\leq i\leq n$, is a partial order. We identify monomials with their exponent vectors throughout, namely identifying  $\bx^\ba =x_1^{a_1}x_2^{a_2}\cdots x_n^{a_n}$ with  $\ba = (a_1, a_2, \ldots, a_n)$ and thus identifying the monomial poset with $\mathbb{N}^n$. The rank function is given by the degree $\deg(\bx^\ba)=a_1+a_2+\cdots+a_n$.
 \end{ex}
 
 More generally, one can associate a monomial poset to any quotient of the polynomial ring by an ideal. If $R$ denotes the polynomial ring and $I$ is an ideal of $R$, a monomial of $S=R/I$ is the coset $\overline{\mu}$  of a monomial  $\mu\in R$ such that  $\mu\not\in I$. Divisibility is defined by $\overline{\mu}$ divides $\overline{\nu}$ if there exists a monomial $\overline{\theta}$ such that $\overline{\nu}=\overline{\mu}\cdot \overline{\theta}$. For the technical details of this paradigm see \cite{Kuz}. In this paper we restrict our attention to a single type of such posets.

\begin{defn}\label{ex: box poset}
The $d_1\times  \ldots \times  d_n$ {\em Clements-Lindstr\"om poset} is the poset of nonzero monomials in the quotient ring 
\[
S=\frac{K[x_1,\ldots, x_n]}{(x_1^{d_1}, \ldots, x_n^{d_n})}.
\]
Identifying monomials with their exponent vectors, the Clements--Lindstr\"om poset is identified with the set
\[
\{\ba=(a_1,\ldots, a_n) \in \N^n \mid 0\leq a_i<d_i \text{ for all } 1\leq i\leq n\}
\]
with the partial order given by componentwise inequality $\ba\leq \bb $ if $a_i\leq b_i$  for all  $1\leq i\leq n$. 
Thus a $d_1\times  \ldots \times  d_n$ Clements--Lindstr\"om poset is the Cartesian product of $n$ paths of with $d_1, d_2, \ldots, d_n$ vertices, respectively. 
Picturing the element $\ba$ of this poset as a unit cube having the corner closest to origin at the lattice point $\ba$ yields the following figures representing  $3\times 4$,  $2\times 2\times 2$ and  $2\times 3\times 4$ Clements--Lindstr\"om posets.
\smallskip

\begin{figure}[htbp]
    \centering
    
    % --- Image 1 (Pink Coordinate Grid) ---
    \begin{minipage}[b]{0.33\textwidth}
        \centering
        \begin{tikzpicture}[scale=1]
            \foreach \x in {0,1,2} {
                \foreach \y in {0,1,2,3} {
                    \filldraw[fill=mypink, draw=black] (\x, \y) rectangle ++(1,1);
                    \node[font=\normalsize] at (\x+0.5, \y+0.5) {$(\x, \y)$};
                }
            }
        \end{tikzpicture}
        \caption*{$3\times 4$}
    \end{minipage}%
    \hspace{-1.3cm}%
        % --- Image 2 (2x2x2 Cube Block) ---
    \begin{minipage}[b]{0.33\textwidth}
        \centering
        % Isometric axes: X is down-left, Y is down-right, Z is up
        \begin{tikzpicture}[x={(-0.7cm,-0.4cm)}, y={(0.7cm,-0.4cm)}, z={(0cm,0.9cm)},scale=0.8]
            % Loop order is back-to-front (0 to 1) to ensure proper overlapping
            \foreach \z in {0,1} {
                \foreach \x in {0,1} {
                    \foreach \y in {0,1} {
                        % Top face
                        \filldraw[fill=mypink!60, draw=black] (\x,\y,\z+1) -- (\x+1,\y,\z+1) -- (\x+1,\y+1,\z+1) -- (\x,\y+1,\z+1) -- cycle;
                        % Left face (facing down-left)
                        \filldraw[fill=mypink, draw=black] (\x+1,\y,\z) -- (\x+1,\y+1,\z) -- (\x+1,\y+1,\z+1) -- (\x+1,\y,\z+1) -- cycle;
                        % Right face (facing down-right)
                        \filldraw[fill=mypink!140, draw=black] (\x,\y+1,\z) -- (\x+1,\y+1,\z) -- (\x+1,\y+1,\z+1) -- (\x,\y+1,\z+1) -- cycle;
                    }
                }
            }
        \end{tikzpicture}
        \caption*{$2 \times 2 \times 2$}
    \end{minipage}%
    \hspace{-1.3cm}%
        % --- Image 3 (2x2x4 Cube Block) ---
    \begin{minipage}[b]{0.33\textwidth}
        \centering
        \begin{tikzpicture}[x={(-0.7cm,-0.4cm)}, y={(0.7cm,-0.4cm)}, z={(0cm,0.9cm)},scale=0.8]
            % Z goes up to 3 for the 4-height block
            \foreach \z in {0,1,2,3} {
                \foreach \x in {0,1} {
                    \foreach \y in {0,1,2} {
                        % Top face
                        \filldraw[fill=mypink!60, draw=black] (\x,\y,\z+1) -- (\x+1,\y,\z+1) -- (\x+1,\y+1,\z+1) -- (\x,\y+1,\z+1) -- cycle;
                        % Left face
                        \filldraw[fill=mypink, draw=black] (\x+1,\y,\z) -- (\x+1,\y+1,\z) -- (\x+1,\y+1,\z+1) -- (\x+1,\y,\z+1) -- cycle;
                        % Right face
                        \filldraw[fill=mypink!140, draw=black] (\x,\y+1,\z) -- (\x+1,\y+1,\z) -- (\x+1,\y+1,\z+1) -- (\x,\y+1,\z+1) -- cycle;
                    }
                }
            }
        \end{tikzpicture}
        \caption*{$2 \times 3 \times 4$}
    \end{minipage}%
\end{figure}
%% Addressing Point C
\begin{figure}[htbp]
    \centering
    \begin{tikzpicture}[scale=1.5]
        % Level 0
        \node (000) at (0,0) {$1$};
        
        % Level 1
        \node (100) at (-1.5,1) {$x_1$};
        \node (010) at (0,1) {$x_2$};
        \node (001) at (1.5,1) {$x_3$};
        
        % Level 2
        \node (110) at (-1.5,2) {$x_1x_2$};
        \node (101) at (0,2) {$x_1x_3$};
        \node (011) at (1.5,2) {$x_2x_3$};
        
        % Level 3
        \node (111) at (0,3) {$x_1x_2x_3$};
        
        % Edges
        \draw (000) -- (100); \draw (000) -- (010); \draw (000) -- (001);
        \draw (100) -- (110); \draw (100) -- (101);
        \draw (010) -- (110); \draw (010) -- (011);
        \draw (001) -- (101); \draw (001) -- (011);
        \draw (110) -- (111); \draw (101) -- (111); \draw (011) -- (111);
    \end{tikzpicture}
    \caption{The Hasse diagram of a $2 \times 2 \times 2$ Clements--Lindstr\"om poset displayed as a distributive lattice.}
\end{figure}
\end{defn}

\subsection{Macaulay posets}

In this section we consider posets endowed with both their partial order, add  an additional {\em total} order and impose a condition on the interaction between the two. In practice we will not need to refer to inequalities with respect to the poset order, thus we utilize the notation $\leq$ for the total order.

\begin{defn}
For a subset $A$ of a poset $\P$ we define the {\bf upper shadow} of $A$ as
\[
\uSdw_\P A=\{b\in P : \text{there is an } a\in A \text{ such that } b \text{ covers } a\}
\]
and the {\bf lower shadow} of $A$ as
\[
\dSdw_\P A =\{b\in P :  \text{there is an } a\in A \text{ such that }  a \text{ covers } b\}.
\]
\end{defn}

\begin{ex} Consider the poset of monomials in $K[x,y]$ ordered by divisibility. In the figure below  the set $A=\{x^2,xy\}$ is shown in red, its lower shadow $\dSdw A=\{x,y\}$ is shown in green and its upper shadow $\uSdw A=\{x^2y,xy^2,y^3\}$ is shown in blue.
	\begin{center}
    \begin{tikzpicture}[scale=0.8]
        \node (0) at (0,0) {1};
        \node (1) at (-0.8,1) {\textcolor{green}{$x$}};
        \node (2) at (0.8,1) {\textcolor{green}{$y$}};
        \node (3) at (-1.6,2) {\textcolor{red}{$x^2$}};
        \node (4) at (0,2) {\textcolor{red}{$xy$}};
        \node (5) at (1.6,2) {$y^2$};
        \node (6) at (-2.4,3) {\textcolor{blue}{$x^3$}};
        \node (7) at (-0.8,3) {\textcolor{blue}{$x^2y$}};
        \node (8) at (0.8,3) {\textcolor{blue}{$xy^2$}};
        \node (9) at (0,4) {$x^2y^2$};
        \node (10) at (-1.6,4) {$x^3y$};
        \node (11) at (-0.8,5) {$x^3y^2$};
        \node (12) at (-3.2,4) {$\cdots$};
        \node (13) at (-2.4,5) {$\cdots$};
        \node (14) at (-1.6,6) {$\cdots$};
        \node (15) at (0,6) {$\cdots$};
        \node (16) at (0.8,5) {$\cdots$};
        \node (17) at (1.6,4) {$\cdots$};
        \node (18) at (2.4,3) {$\cdots$};
     
        \draw [-] (0) -- (1);
        \draw [-] (0) -- (2);
        \draw [-] (1) -- (3);
        \draw [-] (1) -- (4);
        \draw [-] (2) -- (4);
        \draw [-] (2) -- (5);
        \draw [-] (3) -- (6);
        \draw [-] (3) -- (7);
        \draw [-] (4) -- (7);
        \draw [-] (4) -- (8);
        \draw [-] (5) -- (8);
        \draw [-] (9) -- (8);
        \draw [-] (9) -- (7);
        \draw [-] (6) -- (10);
        \draw [-] (10) -- (11);
        \draw [-] (9) -- (11);
        \draw [-] (10) -- (7);
        \draw [-] (6) -- (12);
        \draw [-] (10) -- (13);
        \draw [-] (11) -- (14);
        \draw [-] (11) -- (15);
        \draw [-] (9) -- (16);
        \draw [-] (8) -- (17);
        \draw [-] (5) -- (18);
    \end{tikzpicture}
    \end{center}
\end{ex}

\begin{defn}
Suppose that $\mathcal{P}$ is ranked poset with an additional total order $>$ on it.

A set $A$ consisting of elements of the same rank is called a {\bf segment}, 
if for any $a,b\in A$ and any $c\in \mathcal{P}$ of the same rank such that $a > c > b$,
we must have $c\in A$.
%\end{defn}

%\begin{defn}
The {\bf initial segment} of size $q$ in the $d$-th level $\P_{[d]}$  is 
\[
\Seg_d q= \{ \text{the largest } q \text{ elements with respect to } \leq  \text{ in } \P_{[d]} \}.
\]
%\end{defn}

%\begin{defn}
		The {\bf final segment} of size $q$ in the $d$-th level $\P_{[d]}$ is the set of smallest $q$ elements with respect to $\leq$ of $P_{[d]}$.  
\end{defn}

We define a total order which is often considered when analyzing the poset of monomials in a polynomial ring.
\begin{defn}\label{def: lex}
Let $\bx^\ba=x_1^{a_1}x_2^{a_2}\cdots x_n^{a_n} ,\bx^\bb=x_1^{b_1}x_2^{b_2}\cdots x_n^{b_n}$ be monomials with exponent vectors $\ba=(a_1,\ldots, a_n), \bb=(b_1,\ldots, b_n)$. We say that  $\bx^\ba$ precedes $\bx^\bb$ in the lexicographic order and write $\bx^\ba >_{\lex} \bx^\bb$ if  the leftmost nonzero entry of the vector $\ba-\bb$ is positive. We call this the \textbf{lexicographic} order (lex order for short).
\end{defn}

We use this order to illustrate some of the notions introduced above.

\begin{ex}
Suppose that we consider the poset of monomials of $K[x,y]$ with the lex order.
The set $ \{x^4y^4,x^5y^3,x^6y^2\}$ is a segment (neither initial not final), 
but $\{x^4y^4,x^6y^2\}$ is not a segment.
\end{ex}
    
\begin{ex}
The monomials of degree 2 in $K[x,y,z]$ are ordered with respect to the lexicographic order by
$x^2>_{\lex} xy>_{\lex} xz>_{\lex} y^2>_{\lex} yz>_{\lex} z^2$. Thus we have $\Seg_2 4=\{x^2, xy, xz, y^2\}$.
\end{ex}

\begin{defn}\label{def: Macaulay}
A poset $\P$ is {\bf Macaulay} if there exists an additional total order on $\P$ such that the following hold for any subset $A\subset \P_{[d]}$:
\begin{enumerate}
\item Initial segments have the smallest upper shadows: 
\[
\left | \uSdw_\P\Seg_d |A|  \right | \leq  |\uSdw_\P(A)|;
\]
\item The upper shadow of an initial segment is an initial segment:
\[
 \uSdw_\P  \Seg_d |A| =\Seg_{d+1} |\uSdw_\P\Seg_d|A||.
\]
\end{enumerate}
\end{defn}

The notion of Macaulay poset translates algebraically as follows. A quotient of a polynomial ring by a homogeneous ideal is a {\bf Macaulay ring} if its poset of monomials is a Macaulay poset.

\begin{ex} We work in the poset $\cP$ of monomials in $K[x,y]$  with the partial order given by divisibility and the total order  the lexicographic order. On the left below the set $A=\{x^2,y^2\}$ is pictured in red and its upper shadow is pictured in blue. We have $|A|=2$ and $|\uSdw(A)|=4$. On the right below we take a set $B=\{x^2,xy\}$ consisting of the two lexicographically largest monomials of degree two in $\cP$. The upper shadow of $B$ is also shown in blue. Note that $|\uSdw(B)|=3\leq |\uSdw(A)|=4$ and also that $\uSdw(B)$ is itself an initial segment with respect to the lexicographic order, that is, $\uSdw(B)$ consists of the three lexicographically largest monomials of degree three in $\cP$.

\begin{center}
\begin{tikzpicture}[scale=0.8]
    \node (0) at (0,0) {1};
    \node (1) at (-0.8,1) {$x$};
    \node (2) at (0.8,1) {$y$};
    \node (3) at (-1.6,2) {\textcolor{red}{$x^2$}};
    \node (4) at (0,2) {$xy$};
    \node (5) at (1.6,2) {\textcolor{red}{$y^2$}};
    \node (6) at (-2.4,3) {\textcolor{blue}{$x^3$}};
    \node (7) at (-0.8,3) {\textcolor{blue}{$x^2y$}};
    \node (8) at (0.8,3) {\textcolor{blue}{$xy^2$}};
    \node (9) at (0,4) {$x^2y^2$};
    \node (10) at (-1.6,4) {$x^3y$};
    \node (11) at (-0.8,5) {$x^3y^2$};
    \node (12) at (-3.2,4) {$\cdots$};
    \node (13) at (-2.4,5) {$\cdots$};
    \node (14) at (-1.6,6) {$\cdots$};
    \node (15) at (0,6) {$x^3y^3$};
    \node (16) at (0.8,5) {$x^2y^3$};
    \node (17) at (1.6,4) {$xy^3$};
    \node (18) at (2.4,3) {\textcolor{blue}{$y^3$}};
    \node (19) at (3.2,4) {$\cdots$};
    \node (20) at (2.4,5) {$\cdots$};
    \node (21) at (1.6,6) {$\cdots$};
 
    \draw [-] (0) -- (1);
    \draw [-] (0) -- (2);
    \draw [-] (1) -- (3);
    \draw [-] (1) -- (4);
    \draw [-] (2) -- (4);
    \draw [-] (2) -- (5);
    \draw [-] (3) -- (6);
    \draw [-] (3) -- (7);
    \draw [-] (4) -- (7);
    \draw [-] (4) -- (8);
    \draw [-] (5) -- (8);
    \draw [-] (9) -- (8);
    \draw [-] (9) -- (7);
    \draw [-] (6) -- (10);
    \draw [-] (10) -- (11);
    \draw [-] (9) -- (11);
    \draw [-] (10) -- (7);
    \draw [-] (6) -- (12);
    \draw [-] (10) -- (13);
    \draw [-] (11) -- (14);
    \draw [-] (11) -- (15);
    \draw [-] (9) -- (16);
    \draw [-] (8) -- (17);
    \draw [-] (5) -- (18);
    \draw [-] (17) -- (18);
    \draw [-] (16) -- (17);
    \draw [-] (15) -- (16);
    \draw [-] (18) -- (19);
    \draw [-] (17) -- (20);
    \draw [-] (16) -- (21);
\end{tikzpicture}
\qquad \quad
\begin{tikzpicture}[scale=0.8]
    \node (0) at (0,0) {1};
    \node (1) at (-0.8,1) {$x$};
    \node (2) at (0.8,1) {$y$};
    \node (3) at (-1.6,2) {\textcolor{red}{$x^2$}};
    \node (4) at (0,2) {\textcolor{red}{$xy$}};
    \node (5) at (1.6,2) {$y^2$};
    \node (6) at (-2.4,3) {\textcolor{blue}{$x^3$}};
    \node (7) at (-0.8,3) {\textcolor{blue}{$x^2y$}};
    \node (8) at (0.8,3) {\textcolor{blue}{$xy^2$}};
    \node (9) at (0,4) {$x^2y^2$};
    \node (10) at (-1.6,4) {$x^3y$};
    \node (11) at (-0.8,5) {$x^3y^2$};
    \node (12) at (-3.2,4) {$\cdots$};
    \node (13) at (-2.4,5) {$\cdots$};
    \node (14) at (-1.6,6) {$\cdots$};
    \node (15) at (0,6) {$x^3y^3$};
    \node (16) at (0.8,5) {$x^2y^3$};
    \node (17) at (1.6,4) {$xy^3$};
    \node (18) at (2.4,3) {$y^3$};
    \node (19) at (3.2,4) {$\cdots$};
    \node (20) at (2.4,5) {$\cdots$};
    \node (21) at (1.6,6) {$\cdots$};
 
    \draw [-] (0) -- (1);
    \draw [-] (0) -- (2);
    \draw [-] (1) -- (3);
    \draw [-] (1) -- (4);
    \draw [-] (2) -- (4);
    \draw [-] (2) -- (5);
    \draw [-] (3) -- (6);
    \draw [-] (3) -- (7);
    \draw [-] (4) -- (7);
    \draw [-] (4) -- (8);
    \draw [-] (5) -- (8);
    \draw [-] (9) -- (8);
    \draw [-] (9) -- (7);
    \draw [-] (6) -- (10);
    \draw [-] (10) -- (11);
    \draw [-] (9) -- (11);
    \draw [-] (10) -- (7);
    \draw [-] (6) -- (12);
    \draw [-] (10) -- (13);
    \draw [-] (11) -- (14);
    \draw [-] (11) -- (15);
    \draw [-] (9) -- (16);
    \draw [-] (8) -- (17);
    \draw [-] (5) -- (18);
    \draw [-] (17) -- (18);
    \draw [-] (16) -- (17);
    \draw [-] (15) -- (16);
    \draw [-] (18) -- (19);
    \draw [-] (17) -- (20);
    \draw [-] (16) -- (21);
\end{tikzpicture}
\end{center}
\end{ex}

We end the discussion  with a more refined notion of shadow that incorporates the total order.

\begin{defn}
		Suppose that $\mathcal{P}$ is ranked poset with an additional total order  on it.
		Let $A$ be a segment of some rank and $B$ be all the elements of the same rank that are larger than all the elements of $A$.
		The {\bf new shadow} of $A$ is defined to be
		\begin{align*}
			\uSdw_{\text{new}} (A) = \uSdw(A)\setminus \uSdw(B).
		\end{align*}
\end{defn}

This leads to defining a property of posets termed additivity, which will help strengthen our results. The formal concept of additivity was introduced in \cite{ClementsAdditive} to abstract the combinatorial properties of Clements-Lindstr\"om posets established earlier in \cite{Clements1}.

\begin{defn}
		Suppose that $\mathcal{P}$ is a Macaulay poset.
		We say that $\mathcal{P}$ is {\bf additive} if the following hold:
		\begin{enumerate}
			\item If $A$ is an initial segment and $B$ is a segment such that $|A|=|B|$ then 
			$$|\uSdw_{\text{new}}(A)| \geq |\uSdw_{\text{new}}(B)|.$$
			\item If $B$ is a segment and $C$ is a final segment such that $|B|=|C|$ then 
			$$|\uSdw_{\text{new}}(B)| \geq |\uSdw_{\text{new}}(C)|.$$
		\end{enumerate}
	\end{defn}

\section{Poset operations and algebraic counterparts}\label{s: poset operations}

In this section we collect a number of poset operations we will use in the rest of the paper. The disjoint union operation is  well-known, but the other operations in this section are either being introduced or given formal names here for the first time. 

\subsection{Disjoint union, wedge product, and diamond product}
\begin{defn}\label{def: disjointunion}
Suppose that for $1\leq i\leq t$ we have posets $\P_i$. Their {\bf disjoint union} is the set
$$
\bigsqcup_{i=1}^t \P_i,
$$
meaning that we take the disjoint union of the sets $\P_i$, with the induced partial order $a\leq b$ if and only if $a, b \in \P_i$ and $a\leq b$ in $\P_i$ for some $i$.
\end{defn}

\begin{defn}\label{def: wedge}
Suppose that for $1\leq i\leq t$ we have posets $\P_i$ with unique least element $\ell_i$. Their {\bf wedge product} is the  poset denoted $\bigvee_{i=1}^t \P_i$ or $\P_1\vee \P_2 \vee \cdots \vee  \P_t$ given by
\[
\P_1\vee \P_2 \vee \cdots \vee  \P_t=\left(\bigsqcup_{i=1}^t P_i \right)/ (\ell_1=\ell_2=\cdots  =\ell_t).
\]
Above we take the disjoint union of the sets $\P_i$ in which we identify all the $\ell_i$ into one element,
with the partial order $a\leq b$ if and only if $a\leq b$ in $\P_i$ for some $i$.
\end{defn}

\begin{ex}\label{ex: spider poset}
A {\em spider} poset with legs of length $d_1,\ldots, d_t$ is the wedge product of $t$ path posets of lengths $d_1, d_2, \ldots, d_t$ respectively. Here are some Hasse graphs of spider posets:
\begin{center}
% 1. Spider poset: 1 leg, length 2
\begin{tikzpicture}[scale=0.7, baseline=0pt, every node/.style={circle, draw=black, fill=black, inner sep=1.2pt}]
    \node (root) at (0,0) {};
    
    \node (n1) at (0,1) {};
    \node (n2) at (0,2) {};
    \draw (root) -- (n1) -- (n2);
\end{tikzpicture}
\qquad\quad
% 2. Spider poset: 2 legs, length 3
\begin{tikzpicture}[scale=0.7, baseline=0pt, every node/.style={circle, draw=black, fill=black, inner sep=1.2pt}]
    \node (root) at (0,0) {};
    
    \node (l1) at (-0.8,1) {};
    \node (l2) at (-0.8,2) {};
    % \node (l3) at (-0.8,3) {};
    \draw (root) -- (l1) -- (l2);
    % \draw (root) -- (l1) -- (l2) -- (l3);
    
    \node (r1) at (0.8,1) {};
    \node (r2) at (0.8,2) {};
    \node (r3) at (0.8,3) {};
    \draw (root) -- (r1) -- (r2) -- (r3);
\end{tikzpicture}
\qquad\quad
% 3. Spider poset: 3 legs, length 2
\begin{tikzpicture}[scale=0.7, baseline=0pt, every node/.style={circle, draw=black, fill=black, inner sep=1.2pt}]
    \node (root) at (0,0) {};
    
    \node (l1) at (-1.2,1) {};
    \node (l2) at (-1.2,2) {};
    \node (l3) at (-1.2,3) {};

    \draw (root) -- (l1) -- (l2) -- (l3);
    
    \node (m1) at (0,1) {};
    \node (m2) at (0,2) {};
    \node (m3) at (0,3) {};
    \draw (root) -- (m1) -- (m2) -- (m3);
    
    \node (r1) at (1.2,1) {};
    \node (r2) at (1.2,2) {};
    \node (r3) at (1.2,3) {};
    \draw (root) -- (r1) -- (r2) -- (r3);
\end{tikzpicture}
\qquad\quad
% 4. Asymmetric spider poset: 3 legs, lengths 1, 2, and 3
\begin{tikzpicture}[scale=0.7, baseline=0pt, every node/.style={circle, draw=black, fill=black, inner sep=1.2pt}]
    \node (root) at (0,0) {};
    
    % Leg 1 (length 1)
    \node (l1_1) at (-1.2, 1) {};
    \draw (root) -- (l1_1);
    
    % Leg 2 (length 2)
    \node (l2_1) at (0, 1) {};
    \node (l2_2) at (0, 2) {};
    \draw (root) -- (l2_1) -- (l2_2);
    
    % Leg 3 (length 3)
    \node (l3_1) at (1.2, 1) {};
    \node (l3_2) at (1.2, 2) {};
    \node (l3_3) at (1.2, 3) {};
    \draw (root) -- (l3_1) -- (l3_2) -- (l3_3);
\end{tikzpicture}
\end{center}
\end{ex}

\begin{defn}\label{def: diamond}
Suppose that for $1\leq i\leq t$ we have ranked posets $\P_i$ with unique least element $\ell_i$ and unique largest element $L_i$ such that the posets $\P_i$ have the same maximum rank. Their {\bf diamond product} is the poset denoted $\bigdia_{i=1}^t \P_i$ or $\P_1\di \P_2 \di \cdots \di \P_t$ is given by
\[
\P_1\di \P_2 \di \cdots \di \P_t = \left(\bigsqcup_{i=1}^t \P_i \right)\Big/ (\ell_1=\ell_2=\cdots=\ell_t, L_1=L_2=\cdots=L_t ).
\]
Above we take the disjoint union of the sets $\P_i$ in which we identify all the $\ell_i$ into one element and all the $L_i$ elements into another element,
with the partial order $a\leq b$ if and only if $a\leq b$ in $\P_i$ for some $i$.
\end{defn}

\begin{ex}\label{ex: diamond}\label{ex: tori diamond}
Here are some Hasse graphs of diamond products of path posets:
\begin{center}
\includegraphics[height=2.5cm]{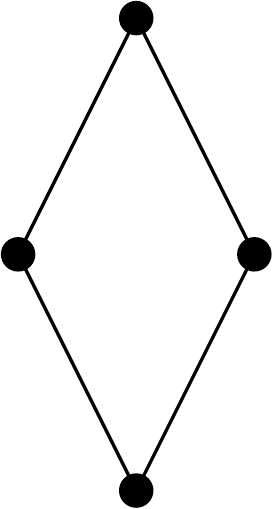}
\quad 
\includegraphics[height=2.5cm]{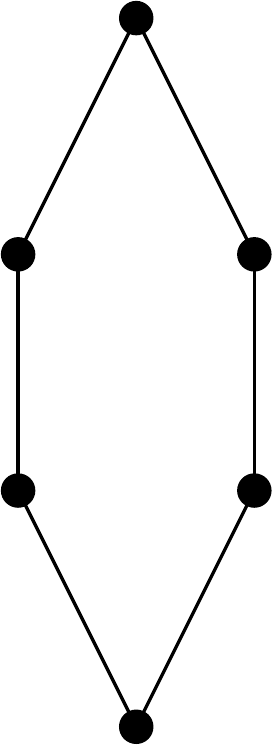}
\quad
\includegraphics[height=2.5cm]{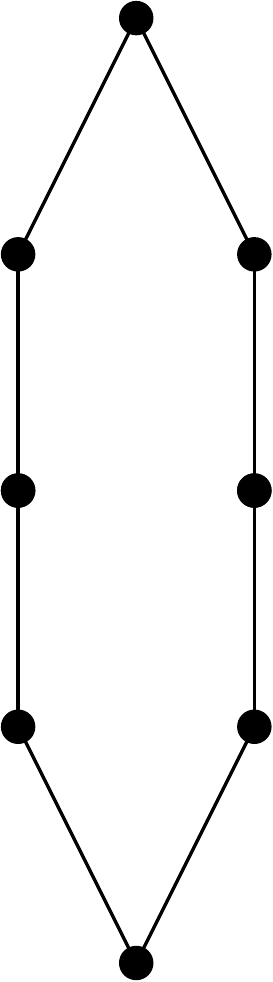}
\quad
\includegraphics[height=2.5cm]{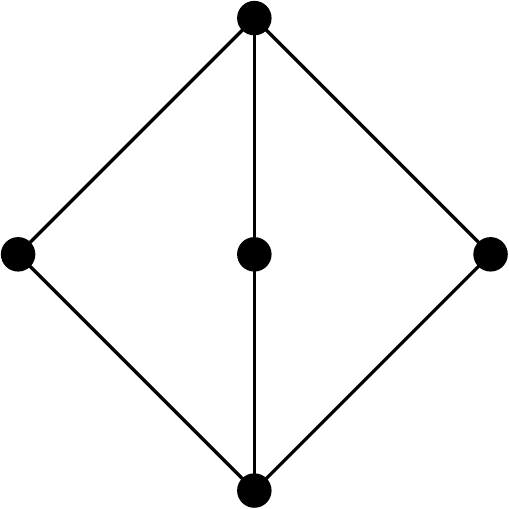}
\end{center}
Diamond products of two path posets of the same length such as the leftmost three posets above are called {\em discrete tori}. The rightmost poset above is called a {\em diamond poset}, which inspired the name for the operation described above.
\end{ex}

\subsection{Fiber product}

One can construct fiber product posets as an operation that generalizes the wedge product, meaning that instead of identifying the least element in the respective posets we identify a set of elements  given by a common subposet $\P_C$.

\begin{defn}\label{def: fiber}
Suppose there are rank-preserving inclusions of posets $\iota_A:\P_C\hookrightarrow \P_A$ and $\iota_B:\P_C\hookrightarrow\P_B$. The {\em fiber product poset} is the set
\[
\P_A\times_{\P_C} \P_B = \P_C \sqcup \left \{ a : a\in \P_A\setminus \iota_A(\P_C)\right \} \sqcup \left \{ b : b \in \P_B\setminus \iota_B(\P_C)\right \}
\]
with order relation 
\begin{itemize}
\item $a\geq c$ iff $a\geq \iota_A(c)$ for $a\in \P_A\setminus \iota_A(\P_C)$ and $c\in \P_C$ and
\item $b\geq c$ iff $b\geq \iota_B(c)$ for $b\in \P_B\setminus \iota_B(\P_C)$ and $c\in \P_C$.
\end{itemize}
\end{defn}

The wedge product of posets $\P_A$ and $\P_B$ can be recovered as $\P_A \vee \P_B = \P_A \times_{*} \P_B$, where $*$ is the singleton set with inclusions $\iota_A: *\to \P_A$ and $\iota_B: *\to \P_B$ mapping $*$ to the least element of $\P_A$ and $\P_B$, respectively.

\begin{ex}
If $\P_A$ and $\P_B$ are both path posets of length $2$ and $\P_C$ is a sub-path of length 1 containing the least element of  $\P_A$ and $\P_B$, 
the resulting fiber product is
\begin{center}
	\includegraphics[height=2.5cm]{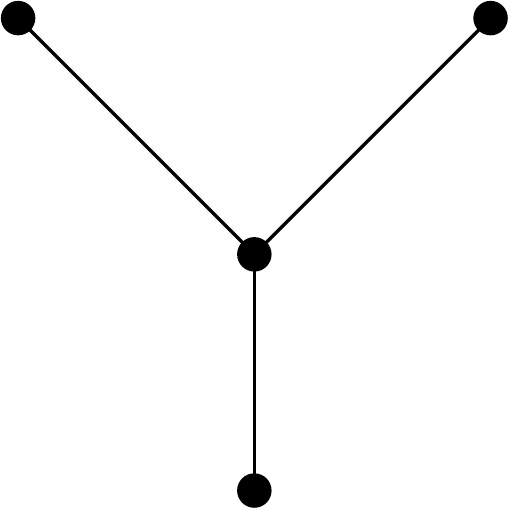}
\end{center}
\end{ex}

The next construction gives a more elaborate example of a fiber product poset.

\begin{defn}\label{def: heart}
Suppose $\P$ is an $a_0\times a_1$ Clements--Lindstr\"om poset and $\Q$ is a $b_0\times b_1$ Clements--Lindstr\"om poset, with a common sub-poset $\L$ that is a $c_0\times c_1$ Clements--Lindstr\"om poset where $c_i = \min\{a_i,b_i\}$. A {\bf heart-shaped poset} is the set
$$
\P \times_\L \Q,
$$
meaning that we take the fiber product of the two 2-dimensional Clements-Lindstr\"om posets over their intersection, with the partial order induced by the fiber product definition.
\end{defn}

The poset introduced in \Cref{def: heart} is depicted in \Cref{heart}, which indicates the origin of its name.
\begin{figure}[h]
\centering
\begin{tikzpicture}[
  roundnode/.style={circle, draw=black, fill=white, thick, minimum size=7mm},
]
    \node at (0,0) {1};
    \node at (-1,1) {x};
    \node at (1,1) {y};
    \node at (-2,2) {.};
    \node at (-2.5,2.5) {.};
    \node at (-3,3) {.};
    \node at (1.7,1.7) {.};
    \node at (2.2,2.2) {.};
    \node at (2.7,2.7) {.};
    \node at (-4,4) {$x^{a_0-1}$};
    \node at (-3.6,4.3) {.};
    \node at (-3.2,4.6) {.};
    \node at (-3.4,4.45) {.};
    \node at (3.5,3.5) {$y^{b_1-1}$};
    \node at (2.9,4) {.};
    \node at (3.1,3.8) {.};
    \node at (2.7,4.2) {.};
    \node at (0,3) {$x^{b_0-1}y^{a_1-1}$};
    \node at (1,3.5) {.};
    \node at (1.3,3.8) {.};
    \node at (1.6,4.1) {.};
    \node at (-1,4) {.};
    \node at (-1.5,4.5) {.};
    \node at (-0.5,3.5) {.};
    \node at (-2,5) {$x^{a_0-1}y^{a_1-1}$};
    \node at (2.5,4.5) {$x^{b_0-1}y^{b_1-1}$};
   % \node at (0,-1) {A heart-shaped Clements-Lindstr\"om poset};

    \draw (-0.2,0.2) -- (-0.8,0.8);
    \draw (0.2,0.2) -- (0.8,0.8);
\end{tikzpicture}
    \caption{A heart-shaped poset}
    \label{heart}
\end{figure}

The name of the fiber product operation is justified by the following algebraic construction: given rings $A,B,C$ and homomorphisms $f:A\to C$ and $g:B\to C$, the fiber product ring is the set
\begin{equation}\label{eq: FP}
A\times_C B=\{(a,b) \mid a\in A, b\in B, f(a)=g(b)\}
\end{equation}
with operations defined componentwise.

The following examples are pertinent to the wedge product and the heart poset, respectively.
\begin{ex}[{\cite[Proposition 3.12]{IMS}}]\label{ex: FP wedge}
    For ideals $I\subset K[x_1,\ldots, x_n]$ and $J\subset K[y_1,\ldots, y_m]$ one has
    \[
    \frac{K[x_1,\ldots, x_n]}{I}\times_K \frac{K[y_1,\ldots, y_m]}{J}=\frac{K[x_1,\ldots, x_n, y_1, \ldots, y_m]}{I+J+(x_iy_j : 1\leq i\leq n, 1\leq j\leq m)}.
    \]
\end{ex}

\begin{ex}\label{ex: FP intersection}
   For a ring $R$ and ideals $I,J$ of $R$ one has
    \[
    \frac{R}{I}\times_{\frac{R}{I+J}} \frac{R}{J} = \frac{R}{I\cap J}.
    \]
    This follows by comparing the exact sequence characterizing the fiber product \cite[(3)]{IMS}
    \[
    0 \to  \frac{R}{I}\times_{\frac{R}{I+J}} \frac{R}{J} \to  \frac{R}{I}\oplus \frac{R}{J} \to  \frac{R}{I+J}\to 0
    \]
    to the standard short exact sequence
    \[
    0 \to  \frac{R}{I\cap J}  \to  \frac{R}{I}\oplus \frac{R}{J} \to  \frac{R}{I+J}\to 0.
    \]
\end{ex}

To explain the connection between the examples above and  poset operations we must first consider how maps of rings induce maps between monomial posets.

\begin{lem}\label{lem: M functor}
Suppose  $R=P/I$ and $S=P/J$ are quotients of a polynomial ring $P$ by monomial ideals $I\subseteq J$ and denote the respective monomial posets $\M_R$ and $\M_S$. The canonical surjection $\pi:R\to S$ gives rise to an injective and rank-preserving poset map $\iota:\M_S\to \M_R$.
\end{lem}
\begin{proof}
    By the definition of $S$, every monomial $\overline{\mu}=\mu+I\in S$ is the image of a unique monomial $\mu=\mu+J\in R$ via $\pi$.  We thus set $\iota(\overline{\mu})=\mu$. It is easy to see that this map preserves ranks. If $\overline{a}, \overline{\mu}, \overline{b}$ are monomials in $S$ satisfying $\overline{a}=\overline{\mu} \overline{b}$, then $a=\mu b$ must hold  in $P$ as otherwise it must be the case that $a-\mu b\in J$ with $J$ a monomial ideal would imply $a, \mu b\in J$, leading to the contradiction $\overline{a}= 0$. Since $a=\mu b$  holds  in $P$, it also holds in $R$. We have thus proved that $\overline{b}\mid \overline{a}$ implies $\iota(\overline{b})\mid \iota(\overline{a})$, indicating that $\iota$ preserves the divisibility order.
\end{proof}

\begin{prop}\label{prop: FP}
\begin{enumerate}
\item If $A$ and $B$ are graded quotients of polynomial rings which have coefficient field $K$,  then
\[
\M_{A\times_K B}=\M_A \vee\M_B.
\]

\item If $I, J$ are  monomial ideals of a polynomial ring $R$,  then
\[
\M_{R/I\times_{R/(I+J)}R/J}=\M_{R/(I\cap J)}=\M_{R/I} \times _{\M_{R/(I+J)}} \M_{R/J}.
\]
\end{enumerate}
\end{prop}

\begin{proof}
Set $A=R/I$ and $B=S/J$. From the presentation of $A\times_K B$ in \Cref{ex: FP wedge} we observe that the monomials of 
$A\times_K B$ are 1, the non-constant monomials of $A$ and the non-constant monomials of $B$, with divisibility relations induced from the posets $\M_A$ and $\M_B$. Thus the first claim follows from \Cref{def: fiber}.

The first equality in assertion (2) follows from \Cref{ex: FP intersection}. Now 
\[
\M_{R/(I\cap J)}=\{m \text{ monomial in } R \mid m\not \in I\cap J\}=\{m \text{ monomial in } R \mid m\not \in I \text{ or } m \not\in J\}.
\]
By \Cref{def: fiber} we have
\begin{eqnarray*}
\M_{R/I} \times _{\M_{R/(I+J)}} \M_{R/J} &=& \left( \M_{R/I} \setminus \iota( \M_{R/(I+J)}) \right) \sqcup  \left( \M_{R/J} \setminus \iota( \M_{R/(I+J)}) \right) \sqcup \M_{R/(I+J)} \\
&=& \{m \text{ monomial in } R \mid m\in J\setminus I\} \\
& & \sqcup \{m \text{ monomial in } R \mid m\in I\setminus J\}\\
& & \sqcup  \{m \text{ monomial in } R \mid m \not \in I+J\} \\
&= &\{m \text{ monomial in } R \mid m\not \in I \text{ or } m \not\in J\} \\ 
&= &\M_{R/(I\cap J)}
\end{eqnarray*}
The third equality above follows since $I, J$ are monomial ideals, therefore the monomials in $I+J$ are the monomials in $I\cup J$. Having identified the second and third posets in claim (2) with the same sub-poset of $\M_R$, the claim is established. 
\end{proof}

From part (1) of \Cref{prop: FP} it follows that the monomial poset of the ring in \Cref{ex: FP wedge} is the wedge product of the monomial posets of the factors. We apply  \Cref{prop: FP} in  settings where the monomial posets involved are Clements--Lindstr\"om posets in \Cref{introthm D} to describe the rings corresponding to wedge products and heart-shaped posets.

\section{The Macaulay property for disjoint unions, wedge, and diamond products}\label{s: uwd}

We now begin our study of how the Macaulay property of a poset transfers under the poset operations introduced in the previous section. 

\subsection{The relationship between disjoint union, wedge and diamond product}
In this section we show that one can study three poset operations with respect to the Macaulay property simultaneously: disjoint union, wedge product and diamond product. The principle which makes this possible is the main result of this subsection, \Cref{prop:equiv union wedge diamond}.

We will use the following notation in the rest of the document:

\begin{notation}
    Let $\P$ be a poset. When applying $\underline{\P}$, we assume $\P$ has a unique smallest element. When applying $\overline{\P}$, we assume $\P$ has a unique largest element.  
    \begin{itemize}[itemsep=1mm]
        \item \(\underline{\P}\) := the poset obtained after removing the smallest element in \(\P\).
       % \vspace{0.25cm}
        \item \(\overline{\P}\) := the poset obtained after removing the largest element in \(\P\).
       % \vspace{0.25cm}
        \item \(\widehat\P\) := the poset obtained by adding a  largest element to \(\P\).
       % \vspace{0.25cm}
        \item \(\uwidehat\P\) := the poset obtained by adding a smallest element to \(\P\).
    \end{itemize}
\end{notation}

The following observations relate the notations described above. 

\begin{lem}\label{lem: hat overline}
\begin{enumerate}
\item\label{prop: hat overline} If $\P$ is a poset, then $\P \cong \overline{\widehat\P}$. If $\P$ is a poset with a unique maximal element, then $\P \cong \widehat{\overline\P}$. 
\item\label{prop: diamond hat wedge}  Let $n\in\mathbb{N}$. For each $i\in\{1, \dots, n\}$, let $\P_i$ be a poset with a unique minimal element. Then $\Diamond_{i=1}^n \widehat P_i \cong \widehat{\bigvee_{i=1}^n \P_i}$.

\end{enumerate}
\end{lem}
The proof of the lemma is omitted as it allows for a straightforward verification. The proof of item (2) is illustrated by \Cref{fig: lem 3.2}.
\begin{figure}[h!]
\centering
    \includegraphics[width=0.8\textwidth]{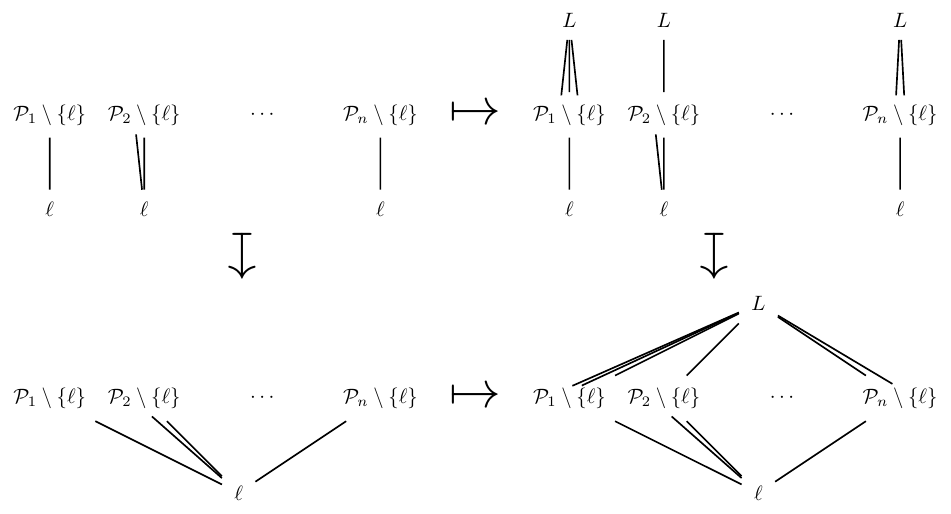}
    \caption{ An illustration of \Cref{lem: hat overline}}
    \label{fig: lem 3.2}
\end{figure}

\begin{prop}\label{prop: hat preserves Macaulay}
    Let $r\in\mathbb{N}$, and let $\P$ be a ranked poset in which every maximal element has rank $r$. Then $\P$ is Macaulay if and only if $\widehat\P$ is Macaulay.
\end{prop}
\begin{proof}
For the backward direction consult \cite[Proposition 6]{BL}.
\begin{comment}
Suppose $<_{\widehat\P}$ is a total ordering of $\widehat\P$ with respect to which $\widehat\P$ is Macaulay. Let $<_\P$ be the restriction of $<_{\widehat\P}$ to $\P$. Then, for any $d<r$ and $A\subseteq \P_d$,
\[
    \lvert\uSdw_\P \Seg_d\lvert A\rvert \rvert = \lvert\uSdw_{\widehat\P} \Seg_d\lvert A\rvert \rvert \leq \lvert \uSdw_{\widehat\P}(A) \rvert = \lvert \uSdw_P(A) \rvert
\]
and
\begin{equation*}
    \uSdw_\P\Seg_d \lvert A\rvert = \uSdw_{\widehat\P}\Seg_d \lvert A\rvert = \Seg_{d+1}\lvert\uSdw_{\widehat\P}(A)\rvert = \Seg_{d+1}\lvert\uSdw_\P(A)\rvert\,.
\end{equation*}
If $A\subseteq \P_r$, then
\begin{equation*}
\lvert \uSdw_\P \Seg_r \lvert A \rvert \rvert = \lvert \varnothing \rvert = \lvert \uSdw_\P(A) \rvert
\end{equation*}
and
\begin{equation*}
\uSdw_\P\Seg_r \lvert A\rvert = \varnothing = \Seg_{r+1}\lvert\uSdw_\P(A)\rvert\
\end{equation*}
Hence $\P$ is Macaulay. 
\end{comment}
Conversely, suppose $<_\P$ is a total ordering of $\P$ with respect to which $\P$ is Macaulay. Let $L$ be the maximal element of $\widehat\P$. Let $<_{\widehat\P}$ be the total ordering of $\widehat\P$ such that, for $p,q\in\P$, we have $p <_\P q$ if and only if $p <_{\widehat\P} q$, and for $p\in\P$, we have $p<L$. Then for any $d<r$ and $A\subseteq \widehat\P_{[d]}$,
\begin{align*}
\lvert \uSdw_{\widehat\P} \Seg_d \lvert A \rvert \rvert &= \lvert \uSdw_\P \Seg_d \lvert A \rvert \rvert \leq \lvert \uSdw_\P(A) \rvert 
= \lvert \uSdw_{\widehat\P}(A) \rvert \\
\uSdw_{\widehat\P}\Seg_d \lvert A\rvert &= \uSdw_\P\Seg_d \lvert A\rvert 
= \Seg_{d+1}\lvert\uSdw_\P(A)\rvert 
= \Seg_{d+1}\lvert\uSdw_{\widehat\P}(A)\rvert.
\end{align*}
If $A \subseteq \widehat\P_{[r + 1]}$, then
\begin{align*}
\lvert \uSdw_{\widehat\P} \Seg_{r + 1} \lvert A \rvert \rvert &= \lvert \varnothing \rvert 
= \lvert \uSdw_{\widehat\P} (A)\rvert \\
\uSdw_{\widehat\P} \Seg_{r + 1} \lvert A \rvert &= \varnothing 
= \Seg_{r+2}\lvert\uSdw_{\widehat\P}(A)\rvert.
\end{align*}
If $\varnothing\neq A\subseteq\widehat\P_{[r]}$, then
\begin{align*}
\lvert\uSdw_{\widehat\P} \Seg_r\lvert A\rvert \rvert &= \lvert\{q\}\rvert 
= \lvert\uSdw_{\widehat\P}(A)\rvert \\ 
\uSdw_{\widehat\P}\Seg_r \lvert A\rvert &= \{q\} = \Seg_{r+1}\lvert\uSdw_{\widehat\P}(A)\rvert.
\end{align*}
We have
\begin{align*}
\lvert\uSdw_{\widehat\P} \Seg_r\lvert\varnothing\rvert \rvert &= \lvert\varnothing\rvert 
= \lvert\uSdw_{\widehat\P}(\varnothing)\rvert\ \\
\uSdw_{\widehat\P}\Seg_r \lvert \varnothing\rvert &= \varnothing 
= \Seg_{r+1}\lvert\uSdw_{\widehat\P}(\varnothing)\rvert.
\end{align*}
Hence $\widehat\P$ is Macaulay.
\end{proof}

\begin{thm}
    \label{prop:equiv union wedge diamond}
Let $\P_1, \ldots, \P_n$ be ranked posets.   
 Whenever the respective constructions are defined, the following are equivalent:
\begin{enumerate}[itemsep=1mm]
    \item $\bigsqcup_{i=1}^n \underline{\P_i}$ is Macaulay, \
    \item $\bigvee_{i=1}^n \P_i$ is Macaulay, \
    \item $\dia_{i=1}^n \widehat{\P_i}$ is Macaulay.
\end{enumerate}
Moreover, the list of properties below are related by implications $(1')\Rightarrow (2')\Rightarrow (3')$
\begin{enumerate}
    \item[(1')] $\bigsqcup_{i=1}^n \P_i$ is Macaulay,
    \item[(2')] $\bigvee_{i=1}^n \P_i$ is Macaulay,
    \item[(3')] $\dia_{i=1}^n \P_i$ is Macaulay.
\end{enumerate}

\end{thm}
\begin{proof}
(3)$\Leftrightarrow$(2) Observe that $\Diamond_{i=1}^n \widehat{\P_i}$ is Macaulay if and only if  $\overline{\Diamond_{i=1}^n \widehat{\P_i}}$ is Macaulay by \Cref{prop: hat preserves Macaulay} since the two posets are related in view of \Cref{lem: hat overline} (1) by
\[
\Diamond_{i=1}^n \widehat{\P_i}\cong \widehat{\overline{\Diamond_{i=1}^n \widehat{\P_i}}}.
\] The isomorphism $\overline{\Diamond_{i=1}^n \widehat{P_i}} \cong \bigvee_{i=1}^n \P_i$, which is a consequence of \Cref{lem: hat overline} (2) completes the argument.

(2)$\Leftrightarrow$(1) By the dual of \Cref{prop: hat preserves Macaulay}, the posets $\bigvee_{i=1}^n \P_i$  and $\underline{\bigvee_{i=1}^n \P_i}$ are simultaneously  Macaulay.
It remains to observe that $\underline{\bigvee_{i=1}^n \P_i} \cong \bigsqcup_{i=1}^n \underline{\P_i}$.

Since there is an isomorphism $\underline{\bigsqcup_{i=1}^n \P_i}\cong  \bigsqcup_{i=1}^n \underline{\P_i}$, we see that $(1')\Rightarrow(1)\Leftrightarrow (2)=(2')$.
We also see that $(2')=(2)\Leftrightarrow(1)\Leftrightarrow (3)$. Finally, we claim that $(3)\Rightarrow (3')$. Towards this end, note that $\overline{\Diamond_{i=1}^n \widehat{\P_i}}$ and $\Diamond_{i=1}^n \P_i$ differ only in the top level. Assuming the poset $\Diamond_{i=1}^n \widehat{\P_i}$ is Macaulay, endow the poset $\Diamond_{i=1}^n \P_i$ with the total obtained by identifying the elements that are not part of its top level with elements of the former poset, extended arbitrarily to incorporate the unique largest element. To check that $\Diamond_{i=1}^n \P_i$ is a Macaulay poset one only needs to check the shadows of elements in level $d-1$, where $d$ is the top rank  of $\Diamond_{i=1}^n \P_i$. But all nonempty sets of elements of rank $d-1$ have an upper shadow of size one in this poset. This concludes the proof.
\end{proof}

\begin{ex}\label{prod of paths}
Since  spider posets are Macaulay by \cite{BE}, it follows using \Cref{prop:equiv union wedge diamond} that diamond products of paths of the same length are Macaulay.
\end{ex}

One should note that the implications $(1')\Rightarrow (2')$ and $(2')\Rightarrow (3')$ in \Cref{prop: hat preserves Macaulay} cannot be reversed.

\begin{ex}
    If $\P$ is the wedge product of a length 1 path with a length 2 path, then $\P\lor\P$ is Macaulay and $\P\sqcup\P$ is not Macaulay. \end{ex}

\begin{ex}\label{ex: diamond not wedge}
    If $\P$ is the $2\times 2$ Clements-Lindstr\"om poset, then $\widehat\P\di\uwidehat\P$ is Macaulay with respect to the union simplicial order (see \Cref{defn: union simplicial order}), but $\widehat\P\lor\uwidehat\P$ is not Macaulay.
\end{ex}

In the next subsection we give conditions which ensure that the posets $\sqcup_{i=1}^n \P_i, \vee_{i=1}^n \P_i$ and $\Diamond_{i=1}^n\P_i$ are simultaneously Macaulay. Specifically,  we will see  that if the $\P_i$ are isomorphic additive posets, %and the total orders are union simplicial (cf. \Cref{defn: union simplicial order})
then the converse holds. See \Cref{thm: equiv union wedge diamond} for a detailed statement.

\subsection{Additivity and iterated poset operations}

       The following definition makes the disjoint union of two posets into a new poset.
        
	\begin{defn}\label{defn: union simplicial order}
		Let $\mathcal{P}_1$ and $\mathcal{P}_2$ be two ranked posets  with total order $\leq_1$ and $\leq_2$.
		Let $\mathcal{Q}$ be the  disjoint union poset of $\mathcal{P}_1$ and $\mathcal{P}_2$.
	Endow $\mathcal{Q}$ with the {\bf union simplicial order} $\leq_{\text{us}}$ such that:
		\begin{enumerate}
			\item If $a,b\in \mathcal{Q}$ such that the rank of $a$ is smaller than the rank of $b$,
			then $a\leq_{\text{us}} b$.
			\item If $a\in \mathcal{P}_1$ and $b\in \mathcal{P}_2$ such that $a$ and $b$ have the same rank,
			then $a \leq_{\text{us}} b$.
			\item If $a,b\in \mathcal{P}_1$ such that $a$ and $b$ have the same rank and $a\leq_1 b$,
			then $a \leq_{\text{us}} b$.
			\item If $a,b\in \mathcal{P}_2$ such that $a$ and $b$ have the same rank and $a\leq_2 b$,
			then $a \leq_{\text{us}} b$.
		\end{enumerate}
		Notice that if the respective operations are defined, the union simplicial order induces a total order on $\mathcal{P}_1 \vee \mathcal{P}_2$ and on $\mathcal{P}_1 \diamond \mathcal{P}_2$, which we refer to as the union simplicial order as well.
	\end{defn}
	
	\begin{rem}
    The union simplicial order on $\bigsqcup_{i=1}^n \P_i$, $\bigvee_{i=1}^n \P_i$, or $\Diamond_{i=1}^n \P_i$  induced by total orders on $\P_1,\P_2,\dots,\P_n$ is defined inductively based on \Cref{defn: union simplicial order}.
\end{rem}

The following proposition, in a form that allows taking the disjoint union of an arbitrary number of copies of a poset is due to Clements \cite[Theorem 1]{Clements} who calls the disjoin union a {\em copy poset}. In the cited paper, Clements attributes the  last statement to personal communication with Bezrukov, without proof. We include a proof  here for completeness. 

\begin{prop}[Clements, {\cite[Theorem 1]{Clements}}]\label{prop: disjoint union of mac posets}
	Let $\P$ be an additive Macaulay poset that satisfies $\uSdw(\P_{[r]})=\P_{[r+1]}$ for all integers $r\geq 0$ and let $\P_1,\ldots, \P_n$ denote  posets isomorphic to $\P$. 	Then  $\P_1\sqcup\ldots \sqcup \P_n$ is Macaulay with respect to the union simplicial order.
    
  Conversely, if $\P_1,\P_2$ are posets isomorphic to $\P$ such that $\P_1\sqcup \P_2$  is Macaulay with the union simplicial order, then $\P$ is additive. 
\end{prop}
\begin{proof}
  We show that if $\P_1\sqcup \P_2$ is Macaulay with the union simplicial order then $\P$ must be additive.  Let $S_2$ be a segment and let $S_1$ be an initial segment of the same size in $\P$. Let $B_1 = S_1$ be an initial segment in $\P_1$, Let $B_2$ be the initial segment that consists of all elements that come after $S_2$ in $\P_2$. Let $C = (B_1 \setminus S_1) \sqcup (B_2 \cup S_2)$. Since $\P_1 \sqcup \P_2$ is Macaulay, $|\uSdw(C)| = |\uSdw(B_1\sqcup B_2)| + |\uSdw_{\text{new}}(S_2)| - |\uSdw_{\text{new}}(S_1)| \leq |\uSdw(B_1 \sqcup B_2)|$. It follows that $|\uSdw_{\text{new}}(S_1)| \geq |\uSdw_{\text{new}}(S_2)|$. This shows the first condition required for additivity. 
 
 Now, let $S_1$ be any segment and $S_2$ be a final segment in $\P$, with $B_1 $  a final segment that consists of all elements that come before $S_1$ in $\P_1$. Let $B_2$=$S_2$ be the final segment in $\P_2$. But then again if we construct C in the same way as above, we have $|\uSdw(C)| = |\uSdw(B_1\sqcup B_2)| + |\uSdw_{\text{new}}(S_2)| - |\uSdw_{\text{new}}(S_1)| \leq |\uSdw(B_1 \sqcup B_2)|$, and again it follows that $|\uSdw_{\text{new}}(S_1)| \geq |\uSdw_{\text{new}}(S_2)|$. This shows the second condition for being additive and completes the proof that $\P$ is additive.
\end{proof}

Summing everything up, we arrive at the equivalence of seven statements listed in \Cref{thm: equiv union wedge diamond}. The following diagram is used to prove that these statements are equivalent below.

\[\begin{tikzcd}[cramped, row sep=2em, column sep=4em]
        & (2) && (5) \\
	(1) \\
	& (3) && (6) \\
        (2) \\
	& (4) && (7) 
	\arrow[Rightarrow,"{\Cref{prop: disjoint union of mac posets}}", from=2-1, to=1-2]
    \arrow[Rightarrow,"\rotatebox{270}{\tiny \Cref{prop:equiv union wedge diamond}}", from=1-2, to=3-2]
	\arrow[Rightarrow, "\rotatebox{270}{\tiny \Cref{prop: disjoint union of mac posets}}", from=4-1, to=2-1]
	\arrow[Rightarrow, "{\Cref{prop: disjoint union of mac posets}}", from=1-2, to=1-4]
	\arrow[Rightarrow,"{\rotatebox{270}{\tiny \Cref{prop:equiv union wedge diamond}}}", from=3-2, to=5-2]
	\arrow[Rightarrow, "{\rotatebox{270}{\tiny \Cref{prop:equiv union wedge diamond}}}", from=3-4, to=5-4]
	\arrow[Rightarrow, "{\rotatebox{270}{\tiny \Cref{prop:equiv union wedge diamond}}}" ,from=1-4, to=3-4]
        \arrow[Rightarrow, "{\text{\Cref{thm: equiv union wedge diamond}}}" ,from=5-4, to=5-2]
	\arrow[Rightarrow, "{\text{\Cref{thm: equiv union wedge diamond}}}",from=5-2, to=4-1]
\end{tikzcd}\]

\begin{thm}\label{thm: equiv union wedge diamond}
Let $\P$ be a Macaulay poset that satisfies $\uSdw(\P_{[r]})=\P_{[r+1]}$ for all integers $r\geq 0$. The following statements are equivalent, where the products in (5), (6), and (7) each have an arbitrary number (at least two) of factors:
\begin{enumerate}
\item $\P$ is additive.
\item $\P\sqcup \P$ is Macaulay with respect to the union simplicial order. %induced by the order on $\P$. 
\item $\P\vee \P$ is Macaulay with respect to the union simplicial order.% induced by the order on $\P$.
\item $\P\di \P$ is Macaulay with respect to the union simplicial order. %induced by the order on $\P$.
\iffalse\end{enumerate}
Moreover, any of the conditions above imply
\begin{enumerate}\fi
\item[(5)] $\P\sqcup \P \sqcup \cdots \sqcup \P$ is Macaulay with respect to the union simplicial order. % induced by the order on $\P$. 
\item[(6)] $\P\vee \P \vee \cdots \vee \P$ is Macaulay with respect to the union simplicial order. % induced by the order on $\P$.
\item[(7)] $\P\di \P \di \cdots \di \P$ is Macaulay with respect to the union simplicial order. % induced by the order on $\P$.
\end{enumerate}
\end{thm}

\begin{proof}
Note that in the proof of \Cref{prop:equiv union wedge diamond}, if the given order on the factors is union simplicial, then the new order is union simplicial. The diagram below shows the overall structure of the proof.

(7)$\implies$(4) Restrict the total order on $\P\di\P\di\cdots\di\P$ to a total order on the product of the last two factors $\P\di\P$. The initial segments of $\P\di\P$ are exactly the initial segments of $\P\di\P\di\cdots\di\P$ that are contained in $\P\di\P$. If $A\subseteq\P\di\P$, then $\uSdw_{\P\di\P} A = \uSdw_{\P\di\P\di\cdots\di\P} A$. So, $\P\di\P$ with this order is Macaulay.

(4)$\implies$(2) Suppose $\P\di\P$ is Macaulay with respect to a union simplicial order induced by a total order on $\P$. This total order induces a union simplicial order on $\P\sqcup\P$. Let $r$ be the rank of $\P$, and let $m$ be the number of elements in level $r-1$ of $\P$. If $r\leq 0$, then everything is Macaulay. So, assume $r>0$.

Suppose $A$ is a subset of level $r-1$ of $\P\sqcup\P$. Then $\lvert\uSdw\Seg_{r-1}\lvert A\rvert\rvert\in\{0,1,2\}$. Suppose $\lvert\uSdw\Seg_{r-1}\lvert A\rvert\rvert = 2$. Then $m < \lvert A\rvert$ because the total order on $\P\sqcup\P$ is union simplicial. Using the pigeonhole principle, it follows that $\lvert\uSdw A\rvert = 2$.  In particular, $\lvert\uSdw\Seg_{r-1}\lvert A\rvert\rvert \leq \lvert\uSdw A\rvert$.

Observe that the remaining conditions for $\P\sqcup\P$ being Macaulay also can be satisfied.
\end{proof}

\begin{rem}
The proof of (4)$\implies$(2) in \Cref{thm: equiv union wedge diamond} can be extended to a product of two different posets $\P\di\mathcal{Q}$ instead of $\P\di\P$, provided that $\P$ and $\mathcal{Q}$  have the same number of elements in level $r-1$.
\end{rem}

In  \Cref{thm: equiv union wedge diamond} we highlighted conditions under which  the Macaulay property of a poset $\P$ can be transferred to the  wedge product $\P\vee \P$. We now investigate conditions under which  the Macaulay property passes conversely from a wedge product to its factors.

\begin{prop}\label{prop: Macaulay wedge product implies Macaulay set}Suppose that $\P, \Q$ are ranked posets having smallest elements such that  
\begin{enumerate}
    \item $\P\vee \Q$ is Macaulay with respect to the total order $\mathcal{O}$.
    \item All elements of rank $1$ of $\Q$ form an initial segment of $\P\vee \Q$. 
\end{enumerate}
Then, both $\P$ and $\Q$ are Macaulay with respect to the total orders $\mathcal{O}|_\P$ and $\mathcal{O}|_\Q$.
\end{prop}

\begin{proof}
Since upper shadows of initial segments are initial segments, it follows by induction on rank that the elements of $\Q$ of each rank form an initial segment of $\P\vee \Q$. Thus in every rank the elements of $\Q$ are larger with respect to the total order than the elements of $\P$. This means that the initial segments of $\Q$ with respect to $\mathcal{O}|_\Q$ are also initial segments of $\P\vee \Q$ with respect to the order $\mathcal{O}$, whence it follows that $\Q$ is Macaulay with the restricted total order. To see that $\P$ is Macaulay, let $d>0$, $A\subset \P_{[d]}$ and set $A'=\Q_{[d]}\cup A$. We deduce
\begin{eqnarray*}
\left | \uSdw_{\P\vee \Q}\Seg_d |A'|  \right | & \leq & |\uSdw_{\P\vee \Q}(A')|\\
 \left | \uSdw_{\P\vee \Q} (\Q_{[d]}\cup \Seg_d |A| )  \right | & \leq & |\Q_{[d+1]}\cup \uSdw_{\P}(A)|\\
  \left | \Q_{[d+1]}\cup  \uSdw_{\P}( \Seg_d |A| )  \right | & \leq & |\Q_{[d+1]}\cup \uSdw_{\P}(A)|\\
\left | \Q_{[d+1]} \right| +\left |  \uSdw_{\P}( \Seg_d |A| )  \right |  & \leq &\left | \Q_{[d+1]} \right| +\left |  \uSdw_{\P}(A) \right|\\
\left |  \uSdw_{\P}( \Seg_d |A| )  \right |  & \leq & \left |  \uSdw_{\P}(A) \right |.
\end{eqnarray*}
\end{proof}

\begin{cor} Suppose that $S=R[x_1,\ldots,x_n]/I$ such that $I$ is a homogeneous ideal whose generators have degree at least three. If $S\times _K S$ is Macaulay, then so is $S$.  
\end{cor}
\begin{proof} Let $S'=R[x_1,\ldots,x_n]/I'$ and $S''=R[y_1,\ldots,y_n]/I''$ be two copies of $S$ with monomial posets $\M'$ and $\M''$, respectively. If we can show that $\{x_1,\ldots,x_n\}$ and $\{y_1,\ldots,y_n\}$ are the only possible initial segments of rank $1$ of $n$ elements of $\M'\vee \M''$, then the claim is obtained by applying \Cref{prop: Macaulay wedge product implies Macaulay set} and \Cref{prop: FP}. 

Indeed, suppose that the initial segment has $d$ elements from $\M'$ and $n-d$ elements of $\M''$ where $0<d<n$. One can compute that the upper shadow for these sets contains $\binom{d}{2}+d+d(n-d)$ and $\binom{n-d}{2} + (n-d) + (n-d)d$ elements in $\M'$ and $\M''$ respectively. Thus, to reach a contradiction with the fact that $\M'\vee \M''$ is Macaulay it suffices to check 
$$\binom{d}{2}+d+d(n-d)+ \binom{n-d}{2} + (n-d) + (n-d)d > \binom{n}{2} + n, $$
which is a straightforward computation.
\end{proof}

\section{The Macaulay property for wedge and diamond products of Clements-Lindstr\"om posets}\label{s: wd box}

Recall that a Clements-Lindstr\"om poset is a standard example of a finite, Macaulay poset. The {\bf dimension} of a Clements--Lindstr\"om poset  is defined to be the number of elements of rank one in the poset.

\subsection{Wedge product of Clements-Lindstr\"om posets}

The main goal of this section is to give a classification for when the wedge product of 2-dimensional Clements-Lindstr\"om posets is Macaulay.    \Cref{prop: wedge box classification} and \Cref{prop: wedge of path and box} achieve this goal.

    The following lemma formalizes the observation that except for the two ranks at the top of a two-dimensional Clements--Lindstr\"om poset, every other monomial has upper shadow 2. The proof is evident.

\begin{lem}\label{prop: size upper shadow of monomial in box}
Consider the Clements-Lindstr\"om poset of monomials in the ring $K[x,y]/(x^a,y^b)$, then we have the following formula for the size of upper shadow of a monomial $t = x^my^n$.
\[|\uSdw(t)| = \begin{cases} 
                    2 & m < a - 1, n < b - 1\\
                    0 & m = a-1 , n = b-1 \\
                    1 & \emph{otherwise}
                \end{cases}
                                    \]
\end{lem}

\begin{lem}\label{prop: size upper shadow segment in box}
Consider the ring $K[x,y]/(x^a,y^b)$ equipped with the lexicographic order. Then we have the following formula for the size of the upper shadow of a segment  $S$ in the monomial poset. Let $k = \lvert S \rvert$. Let $t_1, t_2, \ldots, t_k$ be the monomials in the segment. Then the size of the upper shadow is given by
\[ |\uSdw(S)| = \left(\sum_{i = 1}^{k} |\uSdw(t_i)|\right) - k + 1 \]
\end{lem}
\begin{proof}
The proof is by induction on $k$. The base case is trivial. Let $S = \{t_1, \ldots, t_k\}$. Suppose the formula is true for segments of length up to $k - 1$. Then we have 
\begin{equation*}
    \begin{split}
    |\uSdw(\{t_1,\cdots,t_k\})| &= |\uSdw(\{t_1,\cdots,t_{k-1}\})| + |\uSdw(t_k)| - |\uSdw(\{t_1,\cdots,t_{k-1}\}) \cap \uSdw(t_k)| \\
    &= \left(\sum_{i=1}^{k} \uSdw(t_i)\right) - (k-1) + 1 - |\uSdw(\{t_1,\cdots,t_{k-1}\}) \cap \uSdw(t_k)|
    \end{split}
\end{equation*}
To finish the proof we show $|\uSdw(\{t_1, \ldots, t_{k - 1}\}) \cap \uSdw(t_k)| = 1 $. Suppose $t_{k-1} = x^my^n$ and $m+1 < a, n+2 < b$, then $t_k = x^{m-1}y^{n+1}$ since we are taking segments with the lex order. The upper shadow of $t_{k-1}$ is $\{x^{m+1}y^n,x^my^{n+1}\}$ and the upper shadow of $t_k$ is $\{x^{m}y^{n+1},x^{m-1}y^{n+2}\}$, so the size of their intersection is 1. Similarly we can argue for the 
 cases for other values for $m$ and $n$, and also show that $\uSdw(\{t_1, \ldots, t_{k-2}\}) \cap \uSdw(t_k) = \varnothing$. Finally we use the fact that  $\uSdw(\{t_1, \ldots, t_{k-1}\}) = \uSdw(\{t_1, \ldots, t_{k-2}\}) \cup \uSdw(t_{k - 1})$ to show the result.
\end{proof}

\begin{lem}\label{prop: size new shadow in box}
Consider the ring $K[x,y]/(x^a,y^b)$ with the  lexicographic order, then we have the following formula for the size of the new shadow of a segment $S$ in its monomial poset. 
\[ 
\uSdw_{\text{new}}(S) = 
\begin{cases}
\lvert\uSdw(S)\rvert - 1, & S \emph{ is not an initial segment} \\
\lvert\uSdw(S)\rvert,     & S \emph{ is an initial segment}
\end{cases}
\]
\end{lem}
\begin{proof}
It follows from similar arguments as in the latter part of the previous proposition.
\end{proof}

\begin{thm}\label{thm: wedge 2 boxes different size}
Let $\P_1$ be a $a\times b$ Clements-Lindstr\"om poset and let $\P_2$ be a $c\times d$ Clements-Lindstr\"om poset such that $a \geq c, b \geq d$. Then $\P_1 \vee \P_2$ is Macaulay with respect to the union simplicial order of the two lexicographic orders.
\end{thm}
\begin{proof}
The main idea of the proof is to follow the proof for \Cref{prop: disjoint union of mac posets}, and prove the inequality that depends on the two posets being the same by brute force. We consider $\mathcal{L} = \P_1 \sqcup \P_2$ under the union simplicial order. Let $\P_1^{r},\P_2^{r}$ denote the elements of rank r in the respective posets.

Consider some set $A\subseteq \mathcal{Q}$ with all elements of rank $r$, and let $A_1=A\cap \mathcal{P}_1$ and $A_2 = A\cap \mathcal{P}_2$. We will do a sequence of transformations to change $A$ to an initial segment of the union simplicial order, while not decreasing the shadow at each step. We assume that $A_1 \neq \varnothing$ and $A_2 \neq \varnothing$, since otherwise the problem is handled by symmetry.
	
Now let $B_1$ be the initial segment of size $|A_1|$ in $\mathcal{P}_1$ and
$B_2$ be the initial segment of size $|A_2|$ in $\mathcal{P}_2$.
Put $B=B_1\cup B_2$.
We will show that $|\uSdw(B)| \leq |\uSdw(A)|$. One has,
\begin{align*}
|\uSdw(B)| &= |\uSdw(B_1)| + |\uSdw(B_2)| \tag{By definition of $\mathcal{L}$.},\\
&\leq |\uSdw(A_1)| + |\uSdw(B_2)| \tag{Since $\mathcal{P}_1$ is Macaulay.},\\
&\leq |\uSdw(A_1)| + |\uSdw(A_2)| \tag{Since $\mathcal{P}_2$ is Macaulay.},\\
&= |\uSdw(A)| \tag{By definition of $\mathcal{L}$.}.
\end{align*}
Thus, $|B|=|A|$ and $|\uSdw(B)| \leq |\uSdw(A)|$.
	
Next, we  construct a set $C$ that is an initial segment of the union simplicial order.
Notice that $B_1$ and $B_2$ are initial segments in $\mathcal{P}_1$ and $\mathcal{P}_2$ respectively.
Let $F_2$ be all the elements of rank $r$ in $\mathcal{P}_2$ that are not in $B_2$.
Put $q = \min\{|B_1|, |F_2|\}$.
Let $S_1$ be the first $q$ elements in $B_1$,
and let $S_2$ be the last $q$ elements of $F_2$.
Construct
\begin{equation*}
C = (B_1\setminus S_1) \cup (B_2\cup S_2).
\end{equation*}
Then we have
\begin{align*}
|\uSdw(C)| &= |\uSdw(B_1)| - |\uSdw_{\text{new}}(S_1)| + |\uSdw(B_2)| + |\uSdw_{\text{new}}(S_2)| \tag{Definitions of $\mathcal{Q}$ and $\mathcal{\uSdw_{\text{new}}}$.}\\
&= |\uSdw(B)| + |\uSdw_{\text{new}}(S_2)| - |\uSdw_{\text{new}}(S_1)| \tag{Definition of $\mathcal{Q}$.}.
\end{align*}
We want to show that $|\uSdw_{\text{new}}(S_2)| - |\uSdw_{\text{new}}(S_1)| \leq 0$.

\emph{Case 1:} Suppose $q = |B_1|$. Then $S_1$ is an initial segment, but not $S_2$. Since $|F_2| < |(\P_2)_{[r]}| \leq |(\P_1)_{[r]}|$, so if $q = |(\P_2)_{[r]}|$, then $|F_2| \geq |(\P_2)_{[r]}|$ which cannot be since $|B_2|$ is nonempty. So we have that there is at least one element in $(\P_1)_{[r]}$ that is not in $B_1$. It follows that there is at most one element in $S_1$ which has an upper shadow size of 1, the remaining elements have an upper shadow size of 2 (\Cref{prop: size upper shadow of monomial in box}). Then from \Cref{prop: size upper shadow segment in box} we have
\begin{equation*}
\begin{split}
\lvert\uSdw(S_1)\rvert &= \left(\sum_{i = 1}^{q}|\uSdw(t_i)|\right)- q + 1\\
                       &\geq (1 + 2(q - 1)) - q + 1 \\
                       &\geq q
\end{split} 
\end{equation*}
From \Cref{prop: size new shadow in box}, we get $|\uSdw_{\text{new}}(S_1)| \geq q$. Similarly for $S_2$ we get by \Cref{prop: size upper shadow segment in box} and \Cref{prop: size new shadow in box} that $|\uSdw_{\text{new}}(S_2)| \leq q$.

\emph{Case 2:} Suppose \(q = \lvert F_2 \rvert \), then \( S_2 \) is a final segment and \( S_1 \) is a segment. Now, we use a similar technique and try to show again that \( |\uSdw_{\text{new}} (S_1)| \geq |\uSdw_{\text{new}}(S_2)| \). Observe \( S_1 \) is a segment, but not a final segment nor an initial segment. Therefore, each element of \( S_2 \) will individually contribute to 2 upper shadows (i.e., \( |\uSdw(t_i)| = 2 \) for all \( i \in \{1, \ldots, q\} \)). This means
\begin{equation*}
    \begin{split}
        |\uSdw(S_1)| & = \left(\sum_{i=1}^{q} |\uSdw(t_i)|\right) - q + 1 \\
                     & \geq 2q - q + 1 \\
                     & \geq q + 1
    \end{split} 
\end{equation*}
Now, by \Cref{prop: size new shadow in box}, we have 
\[ |\uSdw_{\text{new}}(S_1)| = |\uSdw(S_1)| - 1 \geq (q + 1) - 1 = q. \]

Since \( S_2 \) is a final segment, in the worst case scenario, we have that \( |\uSdw(S_2)| \leq q + 1 \), as for any \( q \)-element segment, the upper shadow in our Clements-Lindstr\"om poset has size at most \( q + 1 \). Then by \Cref{prop: size new shadow in box}, we have
\[ |\uSdw_{\text{new}}(S_2)| \leq (q + 1) - 1 = q. \]

 Thus, we again have 
\[ |\uSdw_{\text{new}}(S_1)| \geq q \geq |\uSdw_{\text{new}}(S_2)|.\]

	Checking that $C$ is an initial segment of the union simplicial order and that the shadow of an initial segment of the union simplicial order
	is an initial segment of the union simplicial order
	are standard verifications left as an exercise for the interested reader.
	After these exercises are completed we get that $\mathcal{Q}$ is Macaulay with the union simplicial order.

Since the disjoint union is Macaulay, the wedge product is Macaulay by \Cref{prop:equiv union wedge diamond}.
\end{proof}

The converse of \Cref{thm: wedge 2 boxes different size} is not true due to the following counterexample.

\begin{prop}\label{prop: wedge box classification}
    Let $\P_1$ be  a $m\times n$ Clements-Lindstr\"om poset and let  $\P_2$ be a $m'\times n'$ Clements-Lindstr\"om poset such that $m \leq n$, $m' \leq n'$, $m,m' > 1$. 
    Then $\P_1 \vee \P_2$ is Macaulay if and only if $m \leq m'$ and $n \leq n'$ or vice versa.
\end{prop}
\begin{proof}
Let $\P_1$ be the poset of monomials of $K[x_1,x_2]/(x_1^m,x_2^n)$ and let $\P_2$ be the poset of monomials of poset $K[y_1,y_2]/(y_1^{m'},y_2^{n'})$. 
Without loss of generality, let $m \leq m'$. The forward implication follows from \Cref{thm: wedge 2 boxes different size}. We show that if $m < m'$ and $n > n'$ then $\P_1 \vee \P_2$ is not Macaulay. Suppose it is Macaulay. Since $m < m',n,n'$, we must have that $x_1^{m-1}$ is the largest element of rank $m-1$ as every other element of that rank has an upper shadow of size 2, while $x_1^{m-1}$ has an upper shadow of size $1$. There is only one element of rank $m-2$ that has $x_1^{m-1}$ in its upper shadow, which is $x_1^{m-2}$. It follows that this is the largest element of rank $m-2$. By repeating this argument we see that $x_1$ is the largest element of rank $2$. It follows that $\Seg_1 2 = \{ x_1,x_2\} $ by comparing the size of upper shadows of all subsets of size 2 containing $x_1$. Hence  the total order must be the union simplical order  with elements of $\P_1$ larger than elements of $\P_2$. Now consider the upper shadow of the $m'$ rank $n'-1$ elements of $\P_2$. It is easy to see that it has $m'-1$ elements (\Cref{prop: size upper shadow of monomial in box} and \Cref{prop: size upper shadow segment in box}). But the initial segment of length $m'$ has upper shadow of size $m'$. To see this first notice that the size of the upper shadow will be the sum of the sizes of upper shadow of it's two disjoint subsets from $\P_1$ and $\P_2$. Again from \Cref{prop: size upper shadow of monomial in box} and \Cref{prop: size upper shadow segment in box}, we see that the size of the upper shadow is $m + (m' - m) = m' > m' - 1$. This is a contradiction, since initial segments have the smallest upper shadows.
\end{proof}

\begin{prop}\label{prop: wedge of path and box}
    Let $\M_1$ be a path poset of rank $n-1$ and  let $\M_2$ be a $m'\times n'$ Clements-Lindstr\"om poset with $m' \leq n'$. Then $\M_1\vee\M_2$ is Macaulay if and only if $n \leq n'$ or $m' = 1,2$.
\end{prop}

\begin{proof}
($\Leftarrow$) If $n \leq n'$ we are done with the help of \Cref{thm: wedge 2 boxes different size}. If $m' = 1$ we get a spider poset which is Macaulay by \cite{BE}. If $m' = 2$ there are at most 3 elements in each rank and their upper shadows are calculated easily so the claim can be checked.

($\Rightarrow$) Suppose $n > n'$ and $m' > 2$. If $n > m' + n'$, that is, the path has higher rank, then at the highest rank of $\M_2$ the maximal element of $\M_2$ must form an initial segment. This implies that every element of $\M_2$ is larger than every element of $\M_1$ (at each rank), but this is a contradiction since at rank 2 in $\M_2$ every element has an upper shadow of size 2 but in $\M_1$ the element has an upper shadow of size 1. If $n < m' + n'$ then by similar arguments every element of $\M_1$ must be larger than every element of $\M_2$. At rank $n' - 1$ we observe that the set consisting of every element of rank $n'-1$ in $\M_2$ has an upper shadow of size $m' - 1$. The size of this subset is $m'$. Now we estimate the size of the upper shadow of an initial segment. The initial segment has $m'-1$ elements coming from $\M_2$, we have two possible cases here. The first is when the elements in $\M_2$ are two smaller segments of $\M_2$. In this case the size of the upper shadow is $(m' - 1 - k) + k + 1 = m'$. The second is when the elements in $\M_2$ are one segment. Again here the size of the upper shadow is $(m' - 1) + 1 = m'$. This is a contradiction. If $n = m' + n'$ we can take cases and proceed as above. This case is also equivalent to the poset considered in \Cref{lem: path and 2 X k}.
\end{proof}
\begin{ex}
\Cref{prop: wedge of path and box} is not true for disjoint unions. Consider the following example:

\begin{center}
\begin{tikzpicture}[scale=0.7, roundnode/.style={circle, draw=black, fill=white, thick, minimum size=6mm}]

    \node[roundnode] (0) at (0, 0) {0};
    \node[roundnode] (1) at (0, 1.5) {1};
    \node[roundnode] (2) at (0, 3) {2};
    \node[roundnode] (3) at (0, 4.5) {3};
    \node[roundnode] (4) at (0, 6) {4};
    
    \node[roundnode] (5) at (2.25, 0) {5};
    \node[roundnode] (6) at (1.5, 1.5) {6};
    \node[roundnode] (7) at (3, 1.5) {7};
    \node[roundnode] (8) at (1.5, 3) {8};
    \node[roundnode] (9) at (3, 3) {9};
    \node[roundnode] (10) at (2.25, 4.5) {10};

    \draw[-] (0) -- (1);
    \draw[-] (1) -- (2);
    \draw[-] (2) -- (3);
    \draw[-] (3) -- (4);
    
    \draw[-] (5) -- (6);
    \draw[-] (5) -- (7);
    \draw[-] (6) -- (8);
    \draw[-] (7) -- (8);
    \draw[-] (7) -- (9);
    \draw[-] (8) -- (10);
    \draw[-] (9) -- (10);
\end{tikzpicture}
\end{center}
Suppose a Macaulay order did exist. Then node 0 must come first at rank 0. This forces node 1 to come first at rank 1, node 2 at rank 2, node 3 at rank 3, and node 4 at rank 4. However, we should choose node 10 to come first at rank 3 since it has a smaller upper shadow. This is a contradiction.
\end{ex}

\begin{proof}[Proof of \Cref{introthm D} (2)]
By \Cref{ex: FP wedge} the ring in the statement can be written as a fiber product
\[
\frac{K[x_1,x_2]}{(x_1^{a_1}, x_2^{a_2})}\times_K \frac{K[y_1,y_2]}{(y_1^{b_1}, y_2^{b_2})}
\]
and by \Cref{prop: FP} the monomial poset of this ring is a wedge product of a $a_1\times a_2$ Clements--Lindstr\"om poset and a $b_1\times b_2$ Clements--Lindstr\"om poset. The claim follows from \Cref{prop: wedge box classification} and \Cref{prop: wedge of path and box}.

\end{proof}

\subsection{Diamond product of Clements-Lindstr\"om posets}

In this section we classify the Macaulay property for diamond products of Clements--Lindstr\"om posets. We explain the  steps towards the proof of this result.
We begin with the case when $\P$ and $\Q$ have the same dimension which we settle  in \Cref{thm: diamond product of n dimensional box sets}. Then we rule out most of the cases when the dimensions of $\P$ and $\Q$ differ.

\begin{prop} \label{prop: Not Mac test}
Let $\P$ and $\Q$ be $k$-dimensional Clements-Lindstr\"om posets of the same rank $s$ with no side length of $1$. Let $n < s - 1$, $A$ be the set of all monomials in $\P$ of rank $n$, and $B$ be the set of all monomials in $\Q$ of rank $n$. If $|A| < |B|$ and $|\uSdw A| \geq |\uSdw B|$ or vice versa, then $\P \diamond \Q$ is not Macaulay.
\end{prop}

\begin{proof}
Assume for the sake of contradiction that $\P \diamond \Q$ is Macaulay. 
Choose $n$ to be the minimum value that satisfies this condition. Because $\P$ and $\Q$ have dimension $k$, then at rank $1$, $\P$ and $\Q$ both contain $k$ elements. If  a set $S$ contains $i$ elements from $\P$ and $j$ elements from $\Q$, where $i + j = k$, then $|\uSdw S| = (k + (k - 1) + \dots + (k - i)) + (k + (k-1) + \dots + (k - j))$. Then, the minimum upper shadow of $S$ occurs when either $i = k$, or $j = k$. Thus $\Seg_1 k$ is either all elements of $\P$ of rank $1$, or all elements of $\Q$ of rank $1$.  By taking repeated upper shadows of these sets, at some rank $m \leq n$ must be the minimal value in which $A_m$, the set of all elements in $\P$ of rank $m$, must have less elements then $B_m$, the set of all elements in $\Q$ of rank $m$. This implies that $\Seg_{m-1} |A_{m-1}| = A_{m-1}$ as it has the minimal upper shadow of sets containing $|A_{m-1}|$ elements by the argument above. Because upper shadows of initial segments are initial segments it follows that $\Seg_n |A| = A$. If $|\uSdw A| > |\uSdw B|$, then there exists some subset of $B$, $B'$, containing $|A|$ elements, and $|\uSdw B'| \leq |\uSdw B| < |\uSdw A|$ which contradicts that $\P \diamond \Q$ is Macaulay. Thus $|\uSdw A| = |\uSdw B|$. The set $\Seg_n(|A| + 1) = A \cup \{b\} = A'$ for some $b \in B$. Note that $|A'| \leq |B|$, and because $\uSdw (b)$ is disjoint from $\uSdw A$, it implies that $|\uSdw A'| > |\uSdw A| = |\uSdw B|$. We deduce that $|\uSdw \Seg_n (|A| + 1)| > |\uSdw B|$, which contradicts that $\P \diamond \Q$ is Macaulay.
\end{proof}

\begin{thm} \label{thm: diamond product of n dimensional box sets}
    If $\P$ and $\P'$ are $n$-dimensional Clements-Lindstr\"om posets %in which no side of the box is length $1$
    then $\P \diamond \P'$ is Macaulay if and only if $\P$ is isomorphic to $\P'$.
\end{thm}

\begin{proof}
Let $\P$ be an $m_1 \times \cdots \times m_n$ Clements-Lindstr\"om poset and let $\P'$ be a $m_1'\times \cdots\times m'_n$ Clements-Lindstr\"om poset such that $\P$ and $\P'$ are of the same rank $s$. 

If $\P \cong \P'$, then by \Cref{thm: equiv union wedge diamond} $\P \diamond \P'$ is Macaulay. 
    
For the converse consider the contrapositive. Let $\P \ncong \P'$, then there must exist some $m_i \neq m_i'$. Choose $i$ such that $m_i \neq m_i'$ and for all $j < i$, $m_j = m_j'$. Without loss of generality, let $m_i < m_j$. The set $A_{m_i}$, of all monomials in $\P$ of rank $m_i$, must contain less elements then $B_{m_i}$, the set of all monomials in $\P'$ of rank $m_i$. Also, $|A_{s-1}| = |B_{s-1}|=n$ because  $\P$ and $\P'$ are $n$-dimensional. Suppose there is some index $k$ so that $m_i\leq k\leq s-1$ and $|A_k| \geq |B_k|$, then picking the least $k$ with this property gives that $A_{k-1}$ satisfies the requirements of \Cref{prop: Not Mac test}.  On the other hand, if there is no such $k$, then $A_{s-2}$ satisfies the requirements of \Cref{prop: Not Mac test}. Therefore $\P \diamond \P'$ is not Macaulay.
\end{proof}

Next we consider the case of diamond products of Clements-Lindstr\"om posets of different dimensions.

\begin{lem}\label{lem:diff by more than one dim not mac}
    Let $\P$ be an $n$-dimensional Clements-Lindstr\"om poset %with no side length of 1, 
    and $\Q$ an $m$-dimensional Clements-Lindstr\"om poset where $m \leq n - 2$. %with no side length of 1. 
     Then, $\P \diamond \Q$ is not Macaulay. 
\end{lem}
\begin{proof}
Let $\P$ be the monomial poset of  $k[x_1,\dots,x_n]/I$ where $I = (x_1^{d_1},\dots,x_n^{d_n})$. Also, let $\Q$ be the monomial poset of   $k[y_1,\dots,y_m]/J$ where $J = (y_1^{c_1},\dots,y_m^{c_m})$. Assume without loss of generality that $1< d_1 \leq \dots \leq d_n$, and $1< c_1 \leq \dots \leq c_m$. Assume for the sake of contradiction that $\P \diamond \Q$ is Macaulay. First, note that $|\uSdw\{y_1,\dots,y_m\}| \leq m +(m - 1) + \dots + 1$. Also, for any subset of $m$ of the variables $x_i$, the upper shadow of that set must contain at minimum $(n- 1) + (n - 2) + \dots + 2$ elements. Furthermore, any set of $m$ among the $x_i$ and $y_j$ has more elements in its upper shadow than $|\uSdw\{y_1,\dots,y_m\}|$. Thus, $\Seg_1 m = \{y_1,\dots,y_m\}$. Let $s$ be the rank of the maximal element of both $\P$ and $\Q$. By repeatedly taking upper shadows we get that the $t$ elements of rank $s-2$ of $\Q$ form an initial segment  in $\P\diamond \Q$ whose upper shadow has size $m$. %Also, it follows that $|\uSdw \Seg_{s-2} t| = m$. Also, importantly $t \leq m + m-1 + \dots + 1$. 
    
Now consider the set $A$ of all elements of rank $s-2$ in $\P$. Suppose this set contains $r$ elements. Note that $r \geq (n-1) + (n-2) + \dots + 1$ and  also that $|\uSdw A| = n$. Notice that if $m = n - 2$, then $r\geq n-1+m$. 
    
If equality holds, that is $r= n-1+m$, then it follows that all $d_i = 2$ and since $m$ is a positive integer, it follows that $n \geq 3$. In this case, when all $d_i = 2$, then the upper shadow of any collection of $2$ elements in $\P$ must contain at least $3$ distinct elements  from $\uSdw \Seg_{s-2} t$. Thus it follows that $|\uSdw \Seg_{s-2} (t+2)| = m + 3 > n$. Since $\uSdw \Seg_{s-2}( t + 2) \subseteq \uSdw \Seg_{s-2} r$, we deduce that $|\uSdw \Seg_{s-2} r| > n=\uSdw \P_{[s-2]}$. This contradicts that $\P \diamond \Q$ is Macaulay. 
    
From the arguments above, the remaining possibilities are that either $m < n - 2$, or there exists some $d_i \neq 2$. Either way, it follows that $r\geq n+t$. Note that $\Seg_1 (m+1) = \Seg_1 m \cup \{x_i\}$ for some $i$. Then by taking repeated upper shadows, we get that 
\begin{equation*}
\Seg_{s-2} (t + n) = \Seg_{s-2} t \cup \{x_1^{d_1-2}\cdots x_i^{d_i-2}\cdots x_n^{d_n-1}, \ldots, x_1^{d_1-1}\cdots x_i^{d_i-2}\cdots x_n^{d_n-2}\}.
\end{equation*}
Notably, then $|\uSdw \Seg_{s-2} (t + n)| = m + n > n$. Thus, by a similar argument as before, we deduce $|\uSdw \Seg_{s-2} r| > n=\uSdw \P_{[s-2]}$. This contradicts that $\P \diamond \Q$ is Macaulay.
\end{proof}

\begin{lem} \label{lem: n-1 2s}
    Let $\P$ be an $n$-dimensional Clements-Lindstr\"om poset, and $\Q$ a $(n-1)$-dimensional Clements-Lindstr\"om poset for $n > 2$. If $\P \diamond \Q$ is Macaulay, then $\P$ has exactly $n-1$ sides of length $2$ and $\Q$ has no sides of length $2$.
\end{lem}
\begin{proof}
    Let $\P$ be the monomial poset of  $k[x_1,\dots,x_n]/I$ where $I = (x_1^{d_1},\dots,x_n^{d_n})$. Also, let $\Q$ be the monomial poset of $k[y_1,\dots,y_{n-1}]/J$ where $J = (y_1^{c_1},\dots,y_m^{c_{n-1}})$ where for all $d_i > 1$, and $c_j > 1$. 
    Assume for the sake of contradiction that $\P\diamond \Q$ is Macaulay with total order such that the largest element of rank $1$ is in $\Q$. %Let $y_i$ be that maximal element of rank $1$. Then, note that $|\uSdw \{y_1,\dots,y_{n-1}\}| \leq n-1 + \dots + 1$. It is easy to see that this is the minimum upper shadow of a set containing $n$ elements including $y_i$. Thus, 
    As in the proof of \Cref{lem:diff by more than one dim not mac} it follows that $\Seg_1 (n-1) = \{y_1,\dots,y_{n-1}\}$. Let the maximal element of both $\P$ and $\Q$ be rank $s$. Then, by taking repeated upper shadows of the initial segment, we get that $\Seg_{s-2}t$ is the initial segment with $t$ elements where $t$ is the number of elements in $\Q$ of rank $s-2$. Let $A \subseteq (\P \diamond \Q)_{[s-2]}$ be the set of all elements of rank $s-2$ in $\P$. Suppose $A$ contains $r$ elements. Note that $t \leq n-1 + \dots + 1$, and that $r \geq n-1 + \dots + 1$. While this is true, if $r = n-1 + \dots + 1$, then for all $j$,  $d_j = 2$. Because $\P$ and $\Q$ have a maximal element of the same rank, then $d_1 + \dots + d_n - 1 = c_1 + \dots + c_{n-1}$, and because all $c_i > 1$, and $n > 2$, then there exists some $c_i = 2$. Thus, $t \leq n-2 + n-2 + \dots + 1$. Therefore, $t+1 \leq r$ whenever all $d_j = 2$. Whenever all $d_j = 2$, then for any element of rank $s-2$ in $\P$ the upper shadow of that element has $2$ elements distinct from the $\uSdw \Seg_{s-2} t$. Thus, $|\uSdw \Seg_{s-2} (t + 1)| = n+1$. Thus, because $\Seg_{s-2} (t+1) \subseteq \Seg_{s-2} r$, we have  $|\Seg_{s-2}r| > n+1$, but $|\uSdw A| = n$. This contradicts that $\P \diamond \Q$ is Macaulay. Thus, at least one $d_j \neq 2$. If there are at least $2$ values of $j$ such that $d_j \neq 2$, then $r \geq n + n - 1 + n - 3 + \dots + 1$. Thus, $t < r+2$. Also, because any $2$ elements in $\P$ must have an upper shadow of at least $2$ elements which are distinct of $\uSdw \Seg_{s-2} t$ we deduce $|\uSdw \Seg_{s-2} (t+2)| > n+1$. Thus, by similar logic as before, $|\uSdw \Seg_{s-2} r| > n+1$ which contradicts that $\P\diamond \Q$ is Macaulay. 
    
    Alternatively, assume for the sake of contradiction that $\P \diamond \Q$ is Macaulay with total order such that the largest element of rank $1$ is in $\P$. Let the maximal element in rank $1$ be $x_k$. Note that if for all  $i$, $d_i \neq 2$, then it follows that $|\uSdw \{x_k\}| = n$, but $|\uSdw \{y_1\}| = n-1$. Thus, $d_k = 2$. If any $c_j = 2$, then $|\uSdw \{y_j\}| = n-2$, which would contradict that $\P \diamond \Q$ is Macaulay. Thus, for all $j$ we have $c_j \neq 2$. Note that if there exists  $i$ such that $d_i = 2$, and if $m < n-1$, then it follows that 
    \begin{eqnarray*}
        |\uSdw \Seg_{1}(m+1)| &=& n - 1 + \dots + n - 1 - m + n - 1 - m\\
        |\uSdw \{y_1 , \dots , y_{m+1}\}| &=& n - 1 + \dots + n - 1 - m + n - 2 - m.
    \end{eqnarray*}
    This contradicts that $\P \diamond \Q$ is Macaulay. 
    This same logic does not hold for if $m = n - 1$. If for all $i, d_i = 2$, then because $\P$ and $\Q$ are the same rank, it follows that $d_1 + \cdots + d_n - 1 = c_1 + \cdots + c_{n-1}$. Thus, for all $n > 2$,  there must exist a $c_j = 2$, which would contradict that $\P \diamond \Q$ is Macaulay.
\end{proof}

\begin{lem} \label{lem: geq 5 not Mac}
Let $\P$ be a $n$-dimensional Clements-Lindstr\"om poset and $\Q$ a $(n - 1)$-dimensional Clements-Lindstr\"om poset for $n >2$. Then, $\P \diamond \Q$ is not Macaulay.
\end{lem}

\begin{proof}
Let $\P$ be the monomial poset of $k[x_1,\dots,x_n]/I$ where $I = (x_1^{d_1},\dots,x_n^{d_n})$ for $n \geq 5$. Also, let $\Q$ be the monomial poset of $k[y_1,\dots,y_{n-1}]/J$ where $J = (y_1^{c_1},\dots,y_{n-1}^{c_{n-1}})$ where for all $d_i \neq 1$, and $c_j \neq 1$, and $\P$ and $\Q$ are both of rank $s$. 
    
\emph{Case 1:} $n \geq 5$. 
Assume for the sake of contradiction that $\P \diamond \Q$ is Macaulay. Note that by \Cref{lem: n-1 2s}, $n-1$ of $d_1,\dots,d_n$ must be $2$, and also $c_j \neq 2$ for all $j$. Without loss of generality, let $d_1=\cdots=d_{n-1} = 2$, and $d_n \neq 2$. Assume that the maximal element of rank $1$ is $y_j$. Then, it follows as in the proof of \Cref{lem: geq 5 not Mac} that $\Seg_1 n-1 = \{y_1,\dots,y_{n-1}\}$. Note that $\Seg_1 n = \Seg_1 (n-1) \cup \{x_i\}$ for some $i$. Because of the definition of a diamond product, $\uSdw \{x_i\}$ must be distinct from $\uSdw \Seg_1 (n-1)$. Thus, we get that $|\uSdw \Seg_1 n| = n - 1 + \dots + 1 + n - 1$. Notably, if we take the set $A = \{x_1,\dots,x_n\}$ we get that $|\uSdw A| = n-1 + \dots + 1$. This contradicts that $y_j$ is the largest elements of rank $1$.
    
Now, assume that $x_i$ is the largest element of rank $1$.  Then we get that $\Seg_1 n = \{x_1, \dots, x_n\}$ by a similar argument as above. Now, by taking repeated upper shadows, we get that $\P_{[3]} = \Seg_3 t = \{x_1x_2x_3,\dots,x_n^3\}$, where $t$ is the number of elements of rank $3$ in $\P$. 
Using the hypothesis $n\geq 5$, we define a function $f: \Q_{[3]} \to \P_{[3]}$ by $f(y_jy_k) = f(y_j)f(y_k)$ for $j \neq k$, $f(y_j) = x_j$, $f(y_j^2) = x_jx_n$, and $f(y_j^3) = x_jx_n^2$ with the exception that $f(y_1^2y_2) = x_n^3$, $f(y_3^2y_j) = x_jx^2_n$, and $f(y_2^2y_4) = x_3x_n^2$. Note that $f$ is surjective, and because $f(y_1y_4^2) = f(y_1^2y_4)$, then $f$ is not injective. Thus, 
$|\Q_{[3]}| > |\P_{[3]}|$. Also, because $|\P_{[s-1]}| = n$, and $|\Q_{[s-1]}| = n-1$, there must exist some $k$ where $|\P_{[k]}| < |\Q_{[k]}|$, but $|\P_{[k+1]}| \geq |\Q_{[k+1]}|$. Note that $\P_{[k]} = \Seg_k |\P_{[k]}|$. Also,  $A= \Seg_k (|\P_{[k]}| + 1)$ must contain some element in $\Q$ of rank $k$. Because all elements of $\Q$ have a distinct upper shadow from the elements in $\P$, it follows that 
\begin{equation*}
|\uSdw A| > |\uSdw \P_{[k]}| =|\P_{[k+1]}|\geq  |\Q_{[k+1]}| \geq |\uSdw \Q_{[k]}|.
\end{equation*}
Thus, $|\uSdw \Seg_k |\Q_{[k]}|| \geq |\uSdw \Seg_k (|\P_{[k]}| + 1)| > |\uSdw \Q_{[k]}|$. This contradicts that $\P \diamond \Q$ is Macaulay.
      
\emph{Case 2:} $3 \leq n\leq 4$.

Note that the largest element of rank $1$ must be in $\P$. Without loss of generality let the largest element of rank $1$ be $x_1$. Assume that $n = 3$. Then, $\Seg_1 2 = \{x_1,x_2\}$. Assume that $c_1 = 3$, and $c_2 = m$. Because $\P$ and $\Q$ are of the same rank we have $3 + m = 2 + 2 + d_3 - 1$. Thus, $d_3 = m$. Now, by taking repeated upper shadows of $\Seg_1 2$, we get that $\Seg_{m-1} 3 = \{x_1x_2x_3^{m-3},x_1x_3^{m-2},x_2x_3^{m-2}\}$, and that $\uSdw \Seg_{m-1} 3 = \{x_1x_2x_3^{m-2},x_1x_3^{m-1},x_2x_3^{m-1}\}$. Also, if we take the set $A = \{y_1^2y_2^{m-3},y_1y_2^{m-2},y_2^{m-1}\}$, then $\uSdw A = \{y_1^2y_2^{m-2},y_1y_2^{m-1}\}$. This contradicts that $P \diamond Q$ is Macaulay. Thus, $c_1 > 3$. 
 
 Assume that $c_1 \geq 5$, and $c_1 < c_2$. Then we get that $\{y_1^4y_2,y_1^3y_2^2,y_1^2y_2^3,y_1y_2^4,y_2^5\} \subseteq \Q_{[5]}$, and  $\P_{[5]} = \{x_1x_2x_3^3,x_1x_3^4,x_2x_3^4,x_3^5\}$. Then, because $|\P_{[s-1]}| < |\Q_{[s-1]}|$ it follows that $\P \diamond \Q$ is not Macaulay. Thus, if $c_1 = 5$, then $c_2 = 5$. Then, note that $\Q_{[5]} = \{y_1^4y_2,y_1^3y_2^2,y_1^2y_2^3,y_1y_2^4\}$, and $\uSdw \Q_{[5]} = \{y_1^4y_2^2,y_1^3y_2^3,y_1^2y_2^4\}$. Moreover, because $\P$ and $\Q$ are both same rank, then $d_3 = 7$. Therefore, it follows that $\Seg_5 4 = \P_{[5]} = \{x_1x_2x_3^3,x_1x_3^4,x_2x_3^4,x_3^5\}$ and $\uSdw \P_{[5]} = \{x_1x_2x_3^4, x_1x_3^5, x_2x_3^5, x_3^6\}$, but this contradicts that $\P \diamond \Q$ is Macaulay. Thus, $c_1 = 4$. Also, let $c_2 = m$. Because $4 + m = 2 + 2 + d_3 - 1$ it follows that $d_3 = m+1$. Then
 \begin{eqnarray*}
 \P_{[m-1]} &=& \Seg_{m-1} 4 = \{x_1x_2x_3^{m-3},x_1x_3^{m-2},x_2x_3^{m-2},x_3^{m-1}\}\\
 \uSdw \P_{[m-1]} &=&\{x_1x_2x_3^{m-2},x_1x_3^{m-1},x_2x_3^{m-1},x_3^m\}\\
  \Q_{[m-1]} &=& \{y_1^3y_2^{m-4},y_1^2y_2^{m-3},y_1y_2^{m-2},y_2^{m-1}\}\\
  \uSdw \Q_{[m-1]} &=&\{y_1^3y_2^{m-3},y_1^2y_2^{m-2},y_1y_2^{m-1}\}.
  \end{eqnarray*}
This contradicts that $P \diamond Q$ is Macaulay. If $n = 3$, then $P \diamond Q$ is not Macaulay. Now assume that $n = 4$. Note that if $c_3 > 3$, then 
\begin{eqnarray*}
\Q_{[3]} &=&  \{y_1y_2y_3,y_1^2y_2,y_1^2y_3,y_1y_2^2,y_2^2y_3,y_1y_3^2,y_2y_3^2,y_3^3\}\\
\P_{[3]} &=&  \{x_1x_2x_3,x_1x_2x_4,x_1x_3x_4,x_2x_3x_4,x_1x_4^2,x_2x_4^2,x_3x_4^2\}.
\end{eqnarray*}
Then, because $|\Q_{[3]}| > |\P_{[3]}|$, and $|\Q_{[s-1]}| < |\P_{[s-1]}|$ means that $\P \diamond \Q$ is not Macaulay. Thus, $c_1 = c_2 = c_3 = 3$. Thus, because $3 + 3 + 3 = 2 + 2 + 2 + d_4 - 1$, then $d_4 = 2$ which contradicts \Cref{lem: n-1 2s}. Thus, if $n = 4$, then $\P \diamond \Q$ is not Macaulay. Therefore, for $n > 2$, $\P \diamond \Q$ is not Macaulay.
\end{proof}

We have now reduced to the case of Clements-Lindstr\"om posets of dimension $1$ and at most $2$, respectively, which we handle next.

\begin{lem}\label{lem: path and 2 X k}
Let $\P$ be a path poset of rank $n$, and $\Q$ be a Clements-Lindstr\"om poset with the same rank. Then $\P \di \Q$ is Macaulay if and only if $\Q$ is also a path poset or $\Q$ is $2\times k$ for some $k \in \N$.
\end{lem}
\begin{proof}
Let's first prove the backward implication. If $\Q$ is a path poset, then it's trivial. Suppose $\Q$ has dimension $2\times k$ for some $k\in \N$. Then we can show that $\P\di \Q$ is Macaulay with respect to the union simplicial order where the poset elements are larger. With this order, the property that the shadow of an initial segment is an initial segment is obvious. Since in $\P\di \Q$, at each rank $0<r<n$ there are only $3$ elements with number of shadows $1,1,2$, it's not hard to check manually that remaining condition is satisfied.

Now, let's prove the forward implication. First, let's suppose at least three variables are used to construct $\Q$ , or $\Q$ is of the dimension $s\times t$ where $s,t>2$, and $\P \di \Q$ is Macaulay. Then, all elements of rank $1$ in $\Q$ have at least two upper shadows. Then Macaulayness of $\P\di \Q$ makes the element of $\P$ with rank $1$ to be the largest element of $\P\di \Q$ with the rank $1$. Also, Macaulayness of $\P\di \Q$ forces it to have the union simplicial order.  Since all elements of rank $1$ in $Q$ have at least two upper shadows, self-duality of Clements-Lindstr\"om posets implies that for any element of $q\in \Q$ of degree $n-2$, and for any cover $c$ of $q$, there is a $q'\in \Q$ of degree $n-2$ with $q'\ne q$ such that $c$ also covers $q'$. Let $q^*$ be the smallest element of $\Q$ of degree $n-2$. Let $A$ be the set of elements of $\Q$ of rank $n-2$. Since $q^*$ has only one cover, and this cover covers some other element of $A$, we have $\uSdw_{\P\di \Q}(A) = \uSdw_{\P\di \Q}(A-\{q^\ast\})$. Then it is not hard to observe that $\lvert\uSdw_{\P\di \Q}(A)\rvert < \lvert\uSdw_{\P\di \Q}((A- \{q^*\})\cup \{p^*\})\rvert$, where $p^*$ is the element of $\P$ with rank $n-2$. But, this is a contradiction for the Macaulayness. 
\end{proof}

The following is the main result of the section, which encompases the facts proven above.

\begin{thm}\label{thm: classification diamond product of boxes}
Let $\P$ be a $n$-dimensional Clements-Lindstr\"{o}m poset, %with no side length of $1$, 
$\Q$ a $m$-dimensional Clements-Lindstr\"{o}m poset, %with no side length of $1$
and suppose $\P$ and $\Q$ are of the same rank, then $\P \diamond \Q$ is Macaulay if and only if any of the following conditions is true:
\begin{itemize}
\item $\P \cong \Q$
\item $n = 1$, and $m = 2$ where $\Q$ has a side length of $2$
\item $m = 1$, and $n = 2$ where $\P$ has a side length of $2$
\end{itemize}
\end{thm}

\begin{proof}
    This result follows  from \Cref{thm: diamond product of n dimensional box sets}, \Cref{lem: geq 5 not Mac}, 
    \Cref{lem:diff by more than one dim not mac}, and \Cref{lem: path and 2 X k}.
\end{proof}

We next apply \Cref{thm: classification diamond product of boxes} to produce Macaulay rings.

\begin{proof}[Proof of \Cref{introthm D} (1)]\label{proofD1}
We claim that the poset of monomials of the given ring is $\M_A\diamond \M_B$, where $\M_A$ and $\M_B$ are the monomial posets of the rings
\[
A=\frac{K[x_1,\ldots, x_n]}{(x_1^{a_1}, \ldots, x_n^{a_n})} \text{ and } B=\frac{K[y_1,\ldots, y_m]}{(y_1^{b_1}, \ldots, y_m^{b_m})}.
\]
First, observe that by \Cref{prop: FP} and  \Cref{ex: FP wedge} $\M_A\vee \M_B$ is the poset of monomials of
\begin{equation*}
A\times_K B=  
    \frac{K[x_1,\ldots, x_n,y_1,\ldots, y_n]}{(x_1^{a_1},\ldots, x_n^{a_n}, y_1^{b_1},\ldots, y_m^{b_m})+( x_iy_j:1\leq i\leq n, 1\leq j\leq m)}
\end{equation*}
and the ring in the theorem is obtained by identifying elements $x_1^{a_1-1}\cdots x_n^{a_n-1}$ and $y_1^{b_1-1}\cdots y_m^{b_m-1}$ of the ring displayed above,  which correspond to the largest elements of $\M_A$ and $\M_B$, respectively. Thus the poset of monomials of the ring in the theorem is $\M_A\diamond \M_B$. Now the desired conclusion follows from \Cref{thm: classification diamond product of boxes}.
\end{proof}

\section{The Macaulay property for fiber products of Clements-Lindstr\"om  posets}\label{s: fiber box}

In this section we study the {\em heart-shaped posets} introduced in \Cref{def: heart}. These are the fiber product of two 2-dimensional Clements-Lindstr\"om posets over their intersection and have the shape in \Cref{heart}. Our heart-shaped posets are the posets of monomials of certain quotients of $K[x,y]$. In \cite{Abdelfatah} and \cite{He} the authors classify Macaulay rings which are quotients of $K[x,y]$ with respect to the total order being the lex order. The heart-shaped posets are not typically in this list because another order is needed in many cases. In fact, to our knowledge, the posets described below exemplify the first known family where showing different members are Macaulay necessitates distinct total orders.

In \Cref{thm: heart} of this section we characterize the Macaulay heart-shaped posets. The following example shows that not all such posets are Macaulay.

\begin{ex}\label{ex: heart}
Consider in view of \Cref{ex: FP intersection} for  $C=\frac{K[x,y]}{(x^3,y)}=\frac{K[x,y]}{(x^4,y)+(x^3,y^3)}$ the  ring  
\[
\frac{K[x,y]}{(x^4,y)} \times_C \frac{K[x,y]}{(x^3,y^3)}= \frac{K[x,y]}{(x^4,y)\cap(x^3,y^3)} = \frac{K[x,y]}{(x^4,y^3,x^3y)}.
\]
By \Cref{prop: FP}, the monomial poset of the ring $R$ pictured below is the fiber product of two Clements-Lindstr\"om posets corresponding to the monomial posets of each factor glued along the monomial subposet of $C$. The monomials of $C$ are depicted in purple.  One factor in the fiber product consists of the red and purple nodes and the other of the blue and purple nodes. The poset below is not Macaulay because $y^2$ has the smallest upper shadow among elements of degree 3. However, if the poset were Macaulay, $y^2$ would be largest in the total order in rank 3. This forces $xy^2$ to be the largest element of degree $4$ although it has larger upper shadow than $x^3$, leading to a contradiction.

\begin{center}
\begin{tikzpicture}[scale=1]
    \node (0) at (0,0) {\textcolor{violet}{1}};
    \node (1) at (-0.8,1) {\textcolor{violet}{$x$}};
    \node (2) at (0.8,1) {\textcolor{blue}{$y$}};
    \node (3) at (-1.6,2) {\textcolor{violet}{$x^2$}};
    \node (4) at (0,2) {\textcolor{blue}{$xy$}};
    \node (5) at (1.6,2) {\textcolor{blue}{$y^2$}};
    \node (6) at (-2.4,3) {\textcolor{red}{$x^3$}};
    \node (7) at (-0.8,3) {\textcolor{blue}{$x^2y$}};
    \node (8) at (0.8,3) {\textcolor{blue}{$xy^2$}};
   \node (9) at (0,4) {\textcolor{blue}{$x^2y^2$}};

    \draw [-] (0) -- (1);
    \draw [-] (0) -- (2);
    \draw [-] (1) -- (3);
    \draw [-] (1) -- (4);
    \draw [-] (2) -- (4);
    \draw [-] (2) -- (5);
    \draw [-] (3) -- (6);
    \draw [-] (3) -- (7);
    \draw [-] (4) -- (7);
    \draw [-] (4) -- (8);
    \draw [-] (5) -- (8);
   \draw [-] (9) -- (8);
  \draw [-] (9) -- (7);

\end{tikzpicture}

\end{center}
\end{ex}

In the poset $\M$ depicted in \Cref{heart} define sets of nodes $\M_0 = \{x^iy^j \mid j < a_1\}$  and $\M_1 = \{x^iy^j \mid j \geq a_1\}$. Assuming that $a_1\leq b_1$, $\M_0$ is equivalently the Clements-Lindstr\"om poset made up of the left side of the heart in the graph above, and $\M_1$ the remaining nodes. In \Cref{ex: heart} $\M_0$ consists of the red and purple nodes and $\M_1$ of the blue nodes. It is clear that $\M_0 \cup \M_1 = \M$, and $\M_0 \cap \M_1 = \varnothing$. 

Depending on how $a_0$ and $a_1$, respectively how $b_0$ and $b_1$ compare, the posets $\M_0$ and $\M_1$ are Macaulay with respect to the lexicographic order where either $x<_{\lex} y$ or $x>_{\lex} y$. When $\M_0$ and $\M_1$ are Macaulay with respect to different lexicographic orders, the next definition provides a total order which glues together the two different orders and \Cref{thm: heart} gives compatibility conditions on $\M_1$ and $\M_2$ which ensure that the total order introduced below makes the heart poset Macaulay.

\begin{defn}

Define the total order to be used, the {\em twist order}, denoted $\prec$ for a given lexicographic order $<_{\lex}$, as follows. Say nodes $a$ and $b$ have equal rank; then,
\begin{itemize}
    \item If $a,b\in \M_0$, set $a\prec b$ if and only if $a<_{\lex} b$.
    \item If $a,b \in \M_1$, set $a \prec b$ if and only if $b<_{\lex} a$.
    \item If $a\in\M_0,b\in\M_1$, set $a \prec b$.
\end{itemize}    
\end{defn}
\begin{ex}
Below we can see a heart poset with monomials in each rank arranged {\em increasing left to right}, first in lex order with $y >_{\lex} x$, and then using the twist order with $y >_{\lex} x$. The Hasse diagram showing the twist order can be obtained by reflecting that of the lex order over the vertical line drawn at $y^{a_1}$. \Cref{thm: heart} shows that this poset is Macaulay with respect to the twist order. The twist order does not respect multiplication: we have $xy \prec y^2$ but $xy\cdot y=xy^2\succ y^3=y^2\cdot y$. 

\begin{center}
\begin{tikzpicture}[scale=0.8]
    \node at (0,-1) {Using lex order};
    \node (0) at (0,0) {1};
    \node (1) at (-0.8,1) {$x$};
    \node (2) at (0.8,1) {$y$};
    \node (3) at (-1.6,2) {$x^2$};
    \node (4) at (0,2) {$xy$};
    \node (5) at (1.6,2) {$y^2$};
    \node (6) at (-2.4,3) {$x^3$};
    \node (7) at (-0.8,3) {$x^2y$};
    \node (8) at (0.8,3) {$xy^2$};
    \node (9) at (2.4,3) {$y^3$};
    \node (10) at (-3.2,4) {$x^4$};
    \node (11) at (-1.6,4) {$x^3y$};
    \node (12) at (1.6,4) {$xy^3$};
    \node (13) at (3.2,4) {$y^4$};
    \node (14) at (-2.4,5) {$x^4y$};
    \node (15) at (2.4,5) {$xy^4$};

    \draw [-] (0) -- (1);
    \draw [-] (0) -- (2);
    \draw [-] (1) -- (3);
    \draw [-] (1) -- (4);
    \draw [-] (2) -- (4);
    \draw [-] (2) -- (5);
    \draw [-] (3) -- (6);
    \draw [-] (3) -- (7);
    \draw [-] (4) -- (7);
    \draw [-] (4) -- (8);
    \draw [-] (5) -- (8);
    \draw [-] (5) -- (9);
    \draw [-] (6) -- (10);
    \draw [-] (6) -- (11);
    \draw [-] (7) -- (11);
    \draw [-] (8) -- (12);
    \draw [-] (9) -- (12);
    \draw [-] (9) -- (13);
    \draw [-] (10) -- (14);
    \draw [-] (11) -- (14);
    \draw [-] (12) -- (15);
    \draw [-] (13) -- (15);

\end{tikzpicture}
\qquad
\begin{tikzpicture}[scale=0.8]
    \node at (0,-1) {Using twist order};
    \node (0) at (0,0) {1};
    \node (1) at (-0.8,1) {$x$};
    \node (2) at (0.8,1) {$y$};
    \node (3) at (-1.6,2) {$x^2$};
    \node (4) at (0,2) {$xy$};
    \node (5) at (1.6,2) {$y^2$};
    \node (6) at (-2.4,3) {$x^3$};
    \node (7) at (-0.8,3) {$x^2y$};
    \node (8) at (2.4,3) {$xy^2$};
    \node (9) at (0.8,3) {$y^3$};
    \node (10) at (-3.2,4) {$x^4$};
    \node (11) at (-1.6,4) {$x^3y$};
    \node (12) at (1.6,4) {$xy^3$};
    \node (13) at (0,4) {$y^4$};
    \node (14) at (-2.4,5) {$x^4y$};
    \node (15) at (0.8,5) {$xy^4$};

    \draw [-] (0) -- (1);
    \draw [-] (0) -- (2);
    \draw [-] (1) -- (3);
    \draw [-] (1) -- (4);
    \draw [-] (2) -- (4);
    \draw [-] (2) -- (5);
    \draw [-] (3) -- (6);
    \draw [-] (3) -- (7);
    \draw [-] (4) -- (7);
    \draw [-] (4) -- (8);
    \draw [-] (5) -- (8);
    \draw [-] (5) -- (9);
    \draw [-] (6) -- (10);
    \draw [-] (6) -- (11);
    \draw [-] (7) -- (11);
    \draw [-] (8) -- (12);
    \draw [-] (9) -- (12);
    \draw [-] (9) -- (13);
    \draw [-] (10) -- (14);
    \draw [-] (11) -- (14);
    \draw [-] (12) -- (15);
    \draw [-] (13) -- (15);
\end{tikzpicture}
\end{center}
\end{ex}

To prove \Cref{thm: heart} we need a series of lemmas. We include the first without proof, as it is easily deduced. 

\begin{lem}\label{lemma: box segments}
 Consider the monomial poset $\M$ of the ring $R = \frac{K[x,y]}{(x^a,y^b)}$ equipped with a total order $>$. Suppose $A$ is a segment of $\M_{[d]}$ (either initial, final, or neither). Assume that either $b\geq a$ and the total order is the lexicographic order with $x>y$  or $a \geq b$ and the total order is the lexicographic order with $y>x$.  If $A$ is an initial segment of $\M_{[d]}$, then
    \begin{equation*}
    |\nabla_{\mathcal{M}}A| =
    \begin{cases}
        |A|+1, & d\in[0,b-2]\\
        |A|, & d\in[b-1,a+b-2]
    \end{cases}
    \end{equation*}
    If $A$ is a final segment of $\M_{[d]}$, then
    \begin{equation*}
        |\nabla_{\mathcal{M}}A| =
        \begin{cases}
            |A|+1, & d\in[0,a-2]\\
            |A|, & d\in[a-1,a+b-2]
        \end{cases}
    \end{equation*}
    Lastly, if $A$ is a segment that is neither initial nor a final, we have that 
    \[
     |\nabla_{\mathcal{M}}A| = |A| + 1, \, \forall d \in [0,a+b-2].
     \] To get the respective values for the case where either $b\geq a$ and $y>x$  or $a \geq b$ and $x>y$, swap the values for initial and final segments. 
\end{lem}
Now we prove another claim which will help to show that the first requirement for $\M$ to be Macaulay holds for several of our cases. Notice that the second poset which corresponds to the ring $\frac{K[x,y]}{(x^{b_1},y^{b_0})}$ has been twisted, the order being reversed with respect to the standard $b_0\times b_1$ Clements-Lindstr\"om poset by intrechanging the names of the variables while keeping the lexicographic order $y>x$.

\begin{lem}
\label{lemma:initial segments general}
    Say poset $\mathcal{M}$ is the disjoint union of the monomial poset $\P$ of  the ring $L=\frac{K[x,y]}{(x^{a_0},y^{a_1})}$ and  $\Q$ of $R=\frac{K[x,y]}{(x^{b_1},y^{b_0})}$, where $a_0 > b_0$, $b_1 > a_1$, $a_0 \geq b_1$, and $a_0+a_1 \geq b_0+b_1$. %$a_1+b_0 \leq b_1$%.
    Also suppose $d \in [a_1+b_0-1,a_0+a_1-2]$ and let $A \subset \mathcal{P}_{[d]}$ such that $A = A_0 \cup A_1$, where $A_0 \subset \P$ and $A_1 \subset \Q$. We will use the lexicographic order with $y >_{\lex} x$.  Then, 
    \[|\nabla_{\mathcal{M}}\Seg_d |A|| \leq |\nabla_{\mathcal{M}}\Seg_d^{\P}|A_0|| + |\nabla_{\mathcal{M}}\Seg_d^{\Q}|A_1||\]
\end{lem}

\Cref{fig: LR gen} shows the top portion of the heart-shaped poset in \Cref{heart} with monomials in a rank arranged in increasing order left to right, both for the lex and twist order.
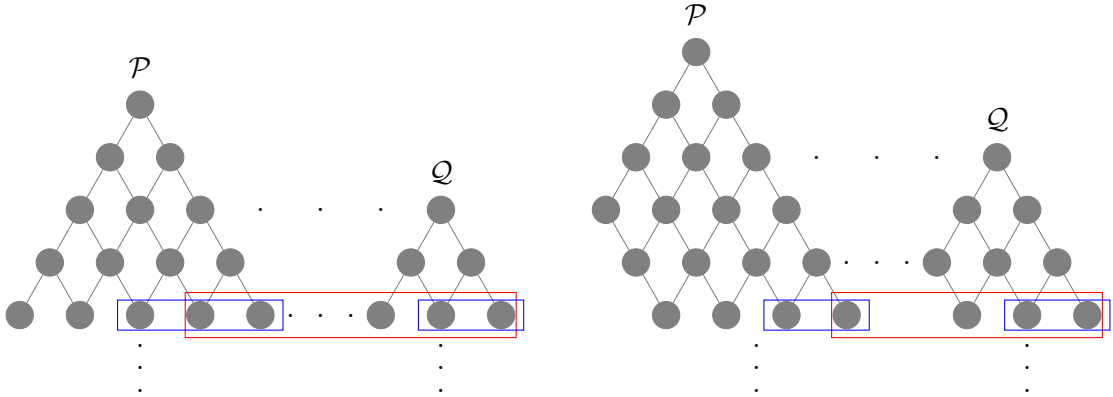
\begin{figure}[h]
\centering
    \begin{tikzpicture}[
    darknode/.style={circle, draw=gray, fill=gray, minimum size=0.5mm},
    ]

    \node[darknode] (0) at (-1.6,0) {};
    \node[darknode] (1) at (-0.8,0) {};
    \node[darknode] (2) at (0,0) {};
    \node[darknode] (3) at (0.8,0) {};
    \node[darknode] (4) at (1.6,0) {};
    \node[darknode] (5) at (-1.2,0.7) {};
    \node[darknode] (6) at (-0.4,0.7) {};
    \node[darknode] (7) at (0.4,0.7) {};
    \node[darknode] (8) at (1.2,0.7) {};
    \node[darknode] (9) at (-0.8,1.4) {};
    \node[darknode] (10) at (0,1.4) {};
    \node[darknode] (11) at (0.8,1.4) {};
    \node[darknode] (12) at (-0.4,2.1) {};
    \node[darknode] (13) at (0.4,2.1) {};
    \node[darknode] (14) at (0,2.8) {};
    \node at (0,-0.4) {.};
    \node at (0,-0.7) {.};
    \node at (0,-1) {.};

    \draw[gray,-] (0) -- (14);
    \draw[gray,-] (4) -- (14);
    \draw[gray,-] (1) -- (13);
    \draw[gray,-] (1) -- (5);
    \draw[gray,-] (2) -- (11);
    \draw[gray,-] (3) -- (8);
    \draw[gray,-] (3) -- (12);
    \draw[gray,-] (2) -- (9);

    \node[darknode] (15) at (3.2,0) {};
    \node[darknode] (16) at (4,0) {};
    \node[darknode] (17) at (4.8,0) {};
    \node[darknode] (18) at (3.6,0.7) {};
    \node[darknode] (19) at (4.4,0.7) {};
    \node[darknode] (20) at (4,1.4) {};

    \draw[gray,-] (15) -- (20);
    \draw[gray,-] (17) -- (20);
    \draw[gray,-] (16) -- (19);
    \draw[gray,-] (16) -- (18);
    \node at (4,-0.4) {.};
    \node at (4,-0.7) {.};
    \node at (4,-1) {.};

    \node at (2,0) {.};
    \node at (2.4,0) {.};
    \node at (2.8,0) {.};
    \node at (1.6,1.4) {.};
    \node at (2.4,1.4) {.};
    \node at (3.2,1.4) {.};

    \node at (0,3.3) {$\P$};
    \node at (4,1.9) {$\Q$};

    \draw[red] (0.6,-0.3) rectangle (5,0.3);
    \draw[blue] (3.7,-0.2) rectangle (5.1,0.2);
    \draw[blue] (-0.3,-0.2) rectangle (1.9,0.2);
    \end{tikzpicture}
    % \caption{An illustration for the proof of \Cref{lemma:initial segments general}}
    % \label{fig: LR}
%\end{figure}
\qquad
%\begin{figure}[h]
    \begin{tikzpicture}[
    darknode/.style={circle, draw=gray, fill=gray, minimum size=0.5mm},
    ]

    \node[darknode] (1) at (-0.8,0) {};
    \node[darknode] (2) at (0,0) {};
    \node[darknode] (3) at (0.8,0) {};
    \node[darknode] (4) at (1.6,0) {};
    \node[darknode] (5) at (-1.2,0.7) {};
    \node[darknode] (6) at (-0.4,0.7) {};
    \node[darknode] (7) at (0.4,0.7) {};
    \node[darknode] (8) at (1.2,0.7) {};
    \node[darknode] (9) at (-0.8,1.4) {};
    \node[darknode] (10) at (0,1.4) {};
    \node[darknode] (11) at (0.8,1.4) {};
    \node[darknode] (12) at (-0.4,2.1) {};
    \node[darknode] (13) at (0.4,2.1) {};
    \node[darknode] (14) at (0,2.8) {};
    \node[darknode] (21) at (-0.4,-0.7) {};
    \node[darknode] (22) at (0.4,-0.7) {};
    \node[darknode] (23) at (1.2,-0.7) {};
    \node[darknode] (24) at (2,-0.7) {};
    \node[darknode] (25) at (3.6,-0.7) {};
    \node[darknode] (26) at (4.4,-0.7) {};
    \node[darknode] (27) at (5.2,-0.7) {};
    \node at (0.8,-1.1) {.};
    \node at (0.8,-1.4) {.};
    \node at (0.8,-1.7) {.};

    \draw[gray,-] (5) -- (14);
    \draw[gray,-] (1) -- (13);
    \draw[gray,-] (21) -- (5);
    \draw[gray,-] (21) -- (11);
    \draw[gray,-] (22) -- (9);
    \draw[gray,-] (22) -- (8);
    \draw[gray,-] (23) -- (12);
    \draw[gray,-] (23) -- (4);
    \draw[gray,-] (24) -- (14);

    \node[darknode] (15) at (3.2,0) {};
    \node[darknode] (16) at (4,0) {};
    \node[darknode] (17) at (4.8,0) {};
    \node[darknode] (18) at (3.6,0.7) {};
    \node[darknode] (19) at (4.4,0.7) {};
    \node[darknode] (20) at (4,1.4) {};

    \draw[gray,-] (15) -- (20);
    \draw[gray,-] (27) -- (20);
    \draw[gray,-] (25) -- (19);
    \draw[gray,-] (25) -- (15);
    \draw[gray,-] (26) -- (18);
    \draw[gray,-] (26) -- (17);
    \node at (4.4,-1.1) {.};
    \node at (4.4,-1.4) {.};
    \node at (4.4,-1.7) {.};

    \node at (2,0) {.};
    \node at (2.4,0) {.};
    \node at (2.8,0) {.};
    \node at (1.6,1.4) {.};
    \node at (2.4,1.4) {.};
    \node at (3.2,1.4) {.};

    \node at (0,3.3) {$\P$};
    \node at (4,1.9) {$\Q$};

    \draw[red] (1.8,-1) rectangle (5.4,-0.4);
    \draw[blue] (4.1,-0.9) rectangle (5.5,-0.5);
    \draw[blue] (0.9,-0.9) rectangle (2.3,-0.5);
    \end{tikzpicture}
    \caption{An illustration for the proof of \Cref{lemma:initial segments general}}
    \label{fig: LR gen}
\end{figure}

\begin{proof}
    This claim is, in simpler words, saying that the initial segments in $\P$ and $\Q$ (shown in blue) whose sizes sum to the size of some initial segment $S$ in the larger poset $\M$ (shown in red) have an upper shadow that is larger than or equal to that of $S$. Say 
    \begin{eqnarray*}
        \Seg_d^{\Q}|A_1| &=& A^{RR},\\
        (\Seg_d|A|\setminus A^{RR})\cap \Q &=& A^{RL},\\ \Seg_d|A|\cap \Seg_d^{\P}|A_0| &=& A^{LR}, \\
        \Seg_d^{\P}|A_0|\setminus A^{LR} &=& A^{LL}.
    \end{eqnarray*}
Note that all of these are pairwise disjoint. We need not consider the cases where $A^{RR}$ or $A^{LR}$ are empty, as those will not be needed for the proof of the theorem. Assume that $A^{RR}$ and $A^{LR}$ are nonempty. We have
\begin{eqnarray*}
\nabla_{\mathcal{M}}\Seg_d|A| &=& \nabla_{\mathcal{M}}A^{RR} \cup \nabla_{\mathcal{M}}A^{RL} \cup \nabla_{\mathcal{M}}A^{LR}\\
\nabla_{\mathcal{M}}\Seg_d^\P|A_0| &=&\nabla_{\mathcal{M}}A^{LR} \cup \nabla_{\mathcal{M}}A^{LL} \\
\nabla_{\mathcal{M}}\Seg_d^\Q|A_1|&=&\nabla_{\mathcal{M}}A^{RR}
\end{eqnarray*}
Also note that $|\nabla_{\mathcal{M}}A^{RR} \cap \nabla_{\mathcal{M}}A^{RL}|=1$, $|\nabla_{\mathcal{M}}A^{LR} \cap \nabla_{\mathcal{M}}A^{LL}|=1$, and 
\begin{equation*}
\{\nabla_{\mathcal{M}}A^{RR} \text{ or } \nabla_{\mathcal{M}}A^{RL}\} \cap \{\nabla_{\mathcal{M}}A^{LR} \text{ or } \nabla_{\mathcal{M}}A^{LL}\} = \varnothing.
\end{equation*}
It is enough to show that $|\nabla_{\mathcal{M}}A^{LL}|-1 \geq |\nabla_{\mathcal{M}}A^{RL}|-1$, or equivalently that $|\nabla_{\mathcal{M}}A^{LL}| \geq |\nabla_{\mathcal{M}}A^{RL}|$.

Note that by definition, $A^{RL}$ is a  final segment and $A^{LL}$ is a (possibly final) segment. If we have that $a_1+b_0 \geq b_1$, then the range given for $d$ in the statement of the lemma then implies that $d \geq b_1-1$ and thus we are dealing with the left side of the poset in \Cref{fig: LR gen}. Thus, by \Cref{lemma: box segments}, we conclude that $|\nabla_{\mathcal{M}}A^{RL}| = |A^{RL}|$. Since $\P$ is in fact the type of Clements-Lindstr\"om poset described in \Cref{lemma: box segments} ($a_0 \geq b_1 > a_1$ and $\P$ has dimensions $a_0 \times a_1$ with order $y>x$), we have that in any case, $|\nabla_{\mathcal{M}}A^{LL}| \geq |A^{LL}| = |A^{RL}| = |\nabla_{\mathcal{M}}A^{RL}|$. 
    
Say $a_1+b_0 < b_1$. Now, if $d \geq b_1-1$, we are in the same case as above, so say $d < b_1-1$. Then, by \Cref{lemma: box segments}, $|\nabla_{\mathcal{M}}A^{RL}| = |A^{RL}| + 1$. Since $a_0 \geq b_1$, it follows that $d \leq a_0-2$, so  we have $|\nabla_{\mathcal{M}}A^{LL}| = |A^{LL}| + 1 = |A^{RL}| + 1 = |\nabla_{\mathcal{M}}A^{RL}|$. 
    
In all cases we have shown $|\nabla_{\mathcal{M}}A^{LL}| \geq |\nabla_{\mathcal{M}}A^{RL}|$, as desired. 
(In fact, the only case in which we could possibly have that $|\nabla_{\mathcal{M}}A^{LL}| < |\nabla_{\mathcal{M}}A^{RL}|$, by \Cref{lemma: box segments}, is if for $A^{LL}$, $d > a_0-2$, $A^{LL}$ was a final segment, and for $A^{RL}$, $d < b_1-2$; but by assumption $(a_0-2,b_1-2)$ is empty, so this cannot happen).
\end{proof}

The following theorem generalizes \Cref{ex: heart}, characterizing Macaulay heart-shaped posets.

\begin{thm}\label{thm: heart}
Let  $\P$ be a $a_0\times a_1$  Clements-Lindstr\"om poset, $\Q$ a $b_0\times b_1$ Clements-Lindstr\"om poset $\Q$, and $\L$   a $c_0\times c_1$ Clements-Lindstr\"om poset, where $c_i = min\{a_i,b_i\}$. This allows us to identify $\L$ with a sub-poset of $\P$ and $\Q$ respectively via  rank-preserving inclusion.  The fiber product poset $\M=\P\times_\L \Q$  is Macaulay if and only if any of the following hold:
\begin{enumerate}
    \item $a_0 \leq b_0$ and $a_1\leq b_1$, in which case $\M \cong \Q$.
    \item $b_0 \leq a_0$ and $b_1\leq a_1$, in which case $\M \cong \P$.
    \item if $b_0<a_0$ and $a_1<b_1$, then 
    
    \begin{inparaenum}
        \item $a_0=b_1$ or
        \item $b_1<a_0$ and $b_1+b_0 \leq a_0 + a_1$ or
        \item $a_0<b_1$ and $a_0+a_1 \leq b_0 + b_1$.
\end{inparaenum}
    \item if $b_0>a_0$ and $a_1>b_1$,  the previous case holds with $a_i$ and $b_i$ interchanged.
\end{enumerate}
\end{thm}

\begin{proof} In case (1), $\L=\P$. It follows from the definition of fiber product that $\M=\P\times_\P \Q=\Q$. Since the poset $\Q$ is Macaulay, the desired conclusion follows.  Case (2) is analogous. For the rest of the proof we assume without loss of generality that $a_0 > b_0$ and $b_1> a_1$ as in (3). 
We claim that this poset $\M$ is Macaulay for any of the following cases:
\begin{enumerate}
        \item[(a)] $a_0=b_1$
        \item[(b)] $b_1<a_0$ and $b_1+b_0 \leq a_0 + a_1$
        \item[(c)] $a_0<b_1$ and $a_0+a_1 \leq b_0 + b_1$
\end{enumerate}
Now notice that if the poset falls into case (c) above, there exists an isomorphism to  a monomial poset which falls into case (b), swapping $x$ and $y$. Similarly, if we have that $a_0=b_1$ and $a_1<b_0$, there exists an isomorphism  to  a monomial poset such that $a_0=b_1$ and $a_1>b_0$. So, we will assume that $a_0 \geq b_1$ and $a_0+a_1 \geq b_0 + b_1$. 

\begin{claim}\label{claim}
   {\em  If $a_1+b_0 > b_1$,  the poset is Macaulay with respect to the lex order with $y >_{\lex} x$. If instead $a_1+b_0 \leq b_1$, the poset is Macaulay with respect to the twist order, also with $y >_{\lex} x$.}
\end{claim}

Note that the largest rank of any monomial $x^iy^j \in \mathcal{M}$ in any of the above cases for poset $\mathcal{M}$ is $a_0+a_1-2$, since otherwise we cannot have that both $i \leq a_0-1$ and $j \leq a_1-1$, or equivalently, that $j \leq b_1-1$ and $i \leq b_0-1$. Thus, we must prove for all ranks $d \in [0,a_0+a_1-2]$ the following two statements, where $A \subset \mathcal{M}_{[d]}$:
\begin{enumerate}
    \item $|\nabla_{\mathcal{M}}\Seg_d|A|| \leq |\nabla_{\mathcal{M}}(A)|$
    \item The upper shadow of every initial segment is an initial segment
\end{enumerate}
First we note that requirement (2) holds for the lex order case by \cite[Proposition 2.5]{MP}. Now we show requirement (2) holds for the twist order case; assume that $a_1+b_0 \leq b_1$. It is easy to see that the order of the monomials in ranks $d \in [0,a_1]$ is the same as when using lex order with $y >_{\lex} x$. Thus, for $A \subset \mathcal{M}_{[d]}$, $A$ is an initial segment w.r.t. the lex order if and only if it is an initial segment w.r.t. the twist order. Since the twist order does not affect the content of upper shadows, we can use the above proof for lex order in this case as well, for $d\in [0,a_1-1]$. Now suppose $d\in [a_1,a_0+a_1-2]$, and let $A \subset \mathcal{M}_{[d]}$ an initial segment. Recall that $\M_0 = \{x^iy^j \vert j < a_1\}$  and $\M_1 = \{x^iy^j \vert j \geq a_1\}$. Let $A \cap \mathcal{M}_0 = A_0$ and $A \cap \mathcal{M}_1 = A_1$. By definition of the twist order, either $A_0 = \varnothing$ or $A_1 = (\mathcal{M}_{1})_{[d]}$.
If $A_0 = \varnothing$, then $A \subset \mathcal{M}_1$, and thus $\nabla_{\mathcal{M}}A \subset \mathcal{M}_1$. Since within $\mathcal{M}_1$, we compare nodes using lex order, the earlier proof holds. So, say that $A_0 \neq \varnothing$, and thus $A_1 =( \mathcal{M}_{1})_{[d]}$. Then, $\nabla_{\mathcal{M}}A= \nabla_{\mathcal{M}_0}A_0 \cup \nabla_{\mathcal{M}_1}A_1 = \nabla_{\mathcal{M}_0}A_0 \cup (\mathcal{M}_1)_{[d+1]}$. By definition of the twist order, this is an initial segment if and only if $\nabla_{\mathcal{M}_0}A_0$ is an initial segment in $\mathcal{M}_0$; we have this statement since the earlier proof holds within $\mathcal{M}_0$, and $A_0$ is an initial segment within $\mathcal{M}_0$.

Let us move on to showing the first requirement for the lex order case; say $a_1+b_0 > b_1$. Then, if $A \subset \mathcal{M}_{[d]}$ for $d \in [0,a_1+b_0-2]$, we have that the poset consisting of the monomials of rank $\leq d$ is isomorphic to that same poset derived from a Clements-Lindstr\"om poset with dimensions $a_0 \times b_1$, which is known to be Macaulay with respect to the lex order $y > x$. So, the requirement holds for this range. Now suppose $d \in [a_0+b_1-1,a_0+a_1-2]$, and let $A \subset \mathcal{M}_{[d]}$. Since we no longer care about the covering relations at ranks below $d$, we can think of $\mathcal{M}$ as the disjoint union of two Clements-Lindstr\"om posets, called $\P$ and $\Q$ respectively, of dimensions $a_0 \times a_1$ and $b_0 \times b_1$. Now,  by our assumption $a_1+b_0 > b_1$ it follows that $a_1+b_0-2 \geq b_1-1$, and since $d > a_1+b_0-2$ excedes the largest degree where the two Clements-Lindstr\"om poset factors in \Cref{thm: heart} intersect, we are in the situation of \Cref{fig: LR gen}. Moreover, we can further assume the second Clements-Lindstr\"om poset has dimensions $b_1 \times b_0$, since if we are concerned with only the subposet consisting of nodes above rank $d$, these posets are isomorphic. 
 (This is convenient for the purpose of the next paragraph, where we are working with the twist order, so that we can conclude in that case using the same logic.) 
This poset is Macaulay for order $y >_{\lex} x$, since $b_1 > b_0$. So, if $A\subset \Q$, we have that requirement (1) holds. Similarly, since $a_0 < a_1$, $\P$ is Macaulay with respect to the chosen order, and thus if $A \subset \P$, requirement (1) also holds. So, instead assume that $A = A_0 \cup A_1$, where $A_0 \subset \P$ and $A_1 \subset \Q$, and neither are empty. By the above logic, we have that 
\begin{eqnarray*}
    |\nabla_{\mathcal{M}}\Seg_d^\P|A_0|| &\leq& |\nabla_{\mathcal{M}}A_0|\\ |\nabla_{\mathcal{M}}\Seg_d^\Q|A_1|| &\leq& |\nabla_{\mathcal{M}}A_1| \text{ and}\\
    \nabla_{\mathcal{M}}A_0 \cap \nabla_{\mathcal{M}}A_1 &=& \varnothing.
\end{eqnarray*}
So, by \Cref{lemma:initial segments general}, 
\[|\nabla_{\mathcal{M}}A| = |\nabla_{\mathcal{M}}A_0| + |\nabla_{\mathcal{M}}A_1| \geq |\nabla_{\mathcal{M}}\Seg_d^\P|A_0|| + |\nabla_{\mathcal{M}}\Seg_d^\Q|A_1|| \geq |\nabla_{\mathcal{M}}\Seg_d|A||.
\]

Now we proceed to the cases in which the twist order is needed; from this point on, we are using the twist order with $y >_{\lex} x$. Again, we can think of poset $\mathcal{M}$ now instead as the disjoint union of posets $\P$ and $\Q$, respectively the Clements-Lindstr\"om posets with dimensions $a_0 \times a_1$ and $b_0 \times b_1$, where the first poset has the lexicographic order with $y >_{\lex} x$ and the second has the lexicographic order with $x >_{\lex} y$, by definition of the twist order.  In order to think of them together as one poset with one order, we can flip the second by interchanging $x$ and $y$, so that the flipped poset  has dimensions $b_1 \times b_0$, with order $y>_{\lex}x$. Now, exactly as above, we get that by \Cref{lemma:initial segments general} that the second requirement holds. 

Now it remains to show the backward direction along the lines of \Cref{ex: heart}. For this it is enough to show that if $a_0 > b_0$, $b_1 > a_1$, and $a_0 > b_1$, but $a_0 + a_1 < b_0 + b_1$, then the resulting poset is not Macaulay. Note that we do not assume that $a_0 \geq b_1$ since in the case that $a_0 = b_1$ but $a_0 + a_1 < b_0 + b_1$ does not hold, we may simply swap the variables $x$ and $y$ and get a poset in which $a_0 + a_1 < b_0 + b_1$ does hold. Assume towards a contradiction that $\mathcal{M}$ is such a poset and it is Macaulay. Then, $|\nabla_{\mathcal{M}}\{y^{b_1-1}\}| = 1$, since $(x)(y^{b_1-1}) \neq 0$, but $(y)(y^{b_1-1})=y^{b_1}=0$. The other elements in rank $b_1-1$, by assumption, all have upper shadow of size 2. So, $y^{b_1-1}$ must be the greatest in its rank according to the total order. By contrast $x^{a_0-1}$, which by similar reasoning also has upper shadow size 1, does not need to be the greatest in its rank, since there are other vertices in that rank which also have upper shadow size 1, such as $x^{a_0-b_1}y^{b_1-1}$. However, we note that $x^{a_0-1}y^{a_1-1}$ has empty upper shadow, while all other in its rank do not, so it must be the greatest in its rank according to the total order. But, it is not a multiple of $y^{b_1-1}$, so starting at rank $b_1-1$, if we repeatedly take the upper shadow of the initial segment of size 1, we see that at rank $a_0+a_1-2$, the resulting subset of $\mathcal{M}_{[a_0+a_1-2]}$ is not an initial segment, since it cannot contain $x^{a_0-1}y^{a_1-1}$. Thus, the requirement that upper shadows of initial segments be initial segments fails. 
\end{proof} 

\begin{proof}[Proof of \Cref{introthm D} (3)] \label{proof D3}
    It follows by \Cref{ex: FP intersection} that for $c_i=\min\{a_i,b_i\}$ and 
    \[
    C=\frac{K[x,y]}{(x^{a_0}, y^{a_1})+(x^{b_0},y^{b_1})}=\frac{K[x,y]}{(x^{c_0},y^{c_1})}
    \]
    the ring in the theorem decomposes as 
   \[
     \frac{K[x,y]}{(x^{a_0}, y^{a_1})\cap(x^{b_0},y^{b_1})} =\frac{K[x,y]}{(x^{a_0}, y^{a_1})}\times_C \frac{K[x,y]}{(x^{b_0},y^{b_1})}.
    \]
    The desired conclusion follows from \Cref{prop: FP} and \Cref{thm: heart}.
\end{proof}

\section{Further conjectures, examples and counterexamples}\label{s: examples} 

In this last section we consider Cartesian products of posets
\begin{equation*}
\P_1 \times \P_2 = \{(a,b) \mid a\in \P_1, b\in \P_2\} 
\end{equation*}
with partial order $(a, b) \leq (a', b')$ if and only if  $a \leq a'$ and $b \leq b'$. An exhaustive search using \cite{BKLS} or \cite{java} yields the following.

\begin{prop}\label{prop: smallest Cartesian product}
\begin{enumerate}
\item The smallest poset that is not Macaulay but decomposes as a nontrivial Cartesian products of Macaulay posets is a Cartesian product of a  path poset of length 1 cross with a Y poset. It has 8 elements. See \Cref{fig:3}.
\item    Let $R = \frac{K[x]}{(x^{2})}$ be a Macaulay ring with poset $\P_{1}$, $S = \frac{K[y, z]}{(y^{2} - z^{2}, y^{3}, z^{3})}$ be a Macaulay ring with poset $\P_{2}$. The poset $\P_{1} \times \P_{2}$ is the smallest (in terms of number of elements) poset that is a nontrivial Cartesian product of two Macaulay posets that come from rings. It has 10 elements and correponds to the ring $R\otimes_KS=\frac{K[x, y,z]}{(x^2, y^{2} - z^{2}, y^{3}, z^{3})}$. See \Cref{fig:4}.
\end{enumerate}
\end{prop}

\begin{figure}[ht!]
\centering 
\includegraphics{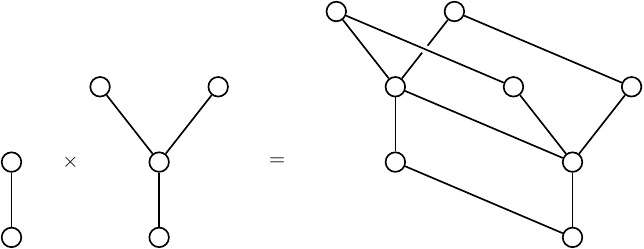}
\caption{The posets for part (1) of \Cref{prop: smallest Cartesian product}}
\label{fig:3}
\end{figure}
\begin{figure}
\centering 
\includegraphics{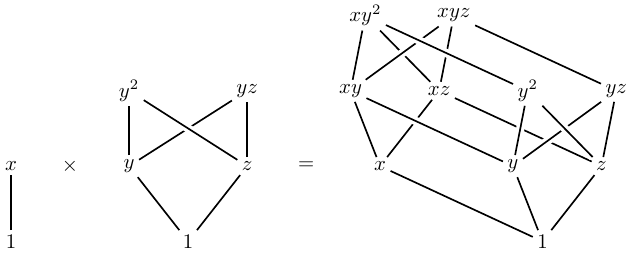}
\caption{The posets for part (2) of \Cref{prop: smallest Cartesian product}}
\label{fig:4} 
\end{figure}

In \cite{MP} it is shown that the tensor product of a Macaulay ring with respect to the lexicographic order and a polynomial ring in a single variable is a Macaulay ring. Via the correspondence between Macaulay rings and posets, this says that the cartesian product of a Macaulay poset that is the monomial poset of a lex-Macaulay ring and an infinite path poset is Macaulay. This prompted the following conjecture, which aims to replace the infinite path with a finite one, as represented by the monomial poset of the ring  $K[x]/(x^n)$.

\begin{conj}[{\cite[Conjecture 6.6.]{Kuz}}]\label{conjecture}
If $S$ is a Macaulay ring, $x$ is a variable not appearing in $S$ and $n>0$ then  $S\otimes_K K[x]/(x^n)$ is a Macaulay ring.
\end{conj}

Using our computational tools, we found  counterexamples to the above conjecture.
\begin{figure}[ht!]
\centering
\includegraphics{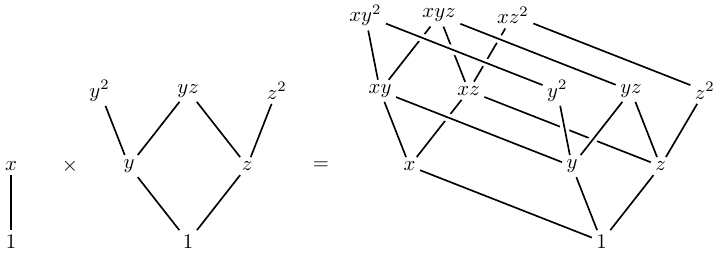}
\caption{A counterexample to {\cite[Conjecture 6.6.]{Kuz}}.}
\label{fig: tensor product counterexample}
\end{figure}

\begin{ex}
\Cref{conjecture} is false. Let $S = \frac{K[y, z]}{(y^{3}, y^{2} z, yz^{2}, z^{3})}$ with monomial poset $\M$, and let $T = \frac{K[z]}{(x^{2})}$ with monomial poset $\P$. By the Macaulay Correspondence Theorem \cite[Theorem 2.6.3]{Kuz}, since $\M \times\P$ is not Macaulay, then $S \otimes_{K} T$ is not Macaulay. See \Cref{fig: tensor product counterexample} for an illustration.
\end{ex}

\noindent However, we found no counterexamples to the following adjusted version of \Cref{conjecture}, which we propose for future investigation.

\begin{conj}\label{conjecture2}
If $S$ is a Macaulay ring, $x$ is a variable not appearing in $S$ and $n$ is strictly larger than the largest degree of any element of $S$. Then  $S\otimes_K K[x]/(x^n)$ is a Macaulay ring.
\end{conj}

\bigskip

\bigskip


\begin{thebibliography}{9999}

\bibitem[Abd]{Abdelfatah} 
A.\,Abedelfatah, {\em Macaulay-Lex rings},
J. Algebra   374 (2013), 122--131.

\bibitem[AAM]{AAM}
H.\,Ananthnarayan, L.~L.\,Avramov and W.~F.\,Moore, {\em Connected sums of Gorenstein local rings}, J. Reine Angew. Math. {\bf 667} (2012), 149--176.

\bibitem[BKLS]{BKLS} 
P.\,Beall, N.\,Kuzmanovski, Y.\,O.\,Li and A.\,Seceleanu {\em The MacaulayPosets package for Macaulay2}, preprint (2025), available at \url{https://arxiv.org/pdf/2510.22843}.

\bibitem[BE]{BE} S.\,L.\, Bezrukov and R.\, Els\"{a}sser, {\em The spider poset is Macaulay}
J. Combin. Theory Ser. A 90 (2000), no. 1, 1--26.%, available at \url{https://www.sciencedirect.com/science/article/pii/S0097316599929941}.

\bibitem[BL]{BL} S.\,L.\, Bezrukov and U.\,Leck {\em Macaulay posets}, Electron. J. Combin. DS12 (2004). %, available at \url{https://www.combinatorics.org/files/Surveys/ds12/ds12v2-2005.pdf}.

\bibitem[BPS]{uwsuper} S.\,L.\, Bezrukov, X.\,Portas, O.\,Serra, {\em A local-global principle for Macaulay posets} Order 16 (1999), no. 1, 57--76.%, available at  \url{http://cs2.uwsuper.edu/sb/Papers/lgp.pdf}.

\bibitem[BD]{java} S.\,L.\, Bezrukov and A.\,Dissanayake, {\em Poset Test for Macaulayness} available at \url{http://cs2.uwsuper.edu/sb/posets/macaulay.php}.

\bibitem[Cho]{Chong} K.\, F.\, E.\,Chong, 
{\em Hilbert functions of colored quotient rings and a generalization of the Clements--Lindstr\"om theorem},
 Journal of algebraic combinatorics 42, 1 (2015), 1--23.

 \bibitem[Cle1]{Clements1} G.\,F.\, Clements, {\em More on the generalized Macaulay theorem. II},
Discrete Math. 18 (1977), no. 3, 253--264.

\bibitem[Cle2]{Clements} G.\,F.\, Clements, 
{\em Characterizing profiles of $k$-families in additive Macaulay posets}, J. Combin. Theory Ser. A80(1997), no.2, 309--319.

\bibitem[Cle3]{ClementsAdditive} G.\,F.\, Clements, 
{\em Additive Macaulay posets},
Order 14 (1997), no. 1, 39--46.

\bibitem[CL]{CL}G.\,Clements and B.\,Lindstr\"om, {\em A generalization of a combinatorial theorem of Macaulay}, J. Combinatorial Theory 7 (1969) 230--238. 


\bibitem[Eng]{Engel} K.\, Engel,  {\em Sperner theory}, Encyclopedia of Mathematics and its Applications, 65, Cambridge Univ. Press, Cambridge, 1997.

\bibitem[M2]{M2}
D.\,Grayson and M.\,Stillman,  {\em Macaulay2, a software system for research in algebraic geometry},
 available at \url{https://faculty.math.illinois.edu/Macaulay2/}.

\bibitem[He]{He} Y.\,He, {\em The classification of Macaulay-Lex ideals in  $k[x,y]$},  Comm. Algebra 40 (2012), no. 3, 897--904.

\bibitem[IMS]{IMS}
A.\,Iarrobino, C.\,McDaniel,A.\,Seceleanu,
{\em Connected Sums of Graded Artinian Gorenstein
Algebras and Lefschetz Properties}, J. Pure Appl. Algebra 226 (2022), no. 1, 106787.

\bibitem[Kat]{Katona} G.\,Katona, {\em A theorem for finite sets}, in: P. Erd\"os, G. Katona (Eds.), Theory of Graphs, Academic Press, New York, 1968, pp. 187--207.  

\bibitem[Kru]{Kruskal}  J.\,Kruskal, {\em The number of simplices in a complex}, in: R. Bellman (Ed.), Mathematical Optimization Techniques, Univ. of California Press, Berkeley, 1963, pp. 251--278.  
 
 \bibitem[Kuz]{Kuz}
N. Kuzmanovski, {\em Macaulay posets and rings}, J. Algebraic Combin. {\bf 62} (2025), no.~4, Paper No. 54, 42 pp.

\bibitem[Mac]{Mac}  F.\, S.\, Macaulay, {\em Some Properties of Enumeration in the Theory of Modular Systems}, Proc. London Math. Soc. (2) {\bf 26} (1927), 531--555.

\bibitem[MM]{MerminMurai} J.\,Mermin and S.\,Murai,
{\em Betti numbers of lex ideals over some Macaulay-Lex rings}, 
J. Algebraic Combin. 31 (2010), no. 2, 299--318.

\bibitem[Mer]{Mermin} J.\,Mermin
{\em Lexlike sequences}, 
J. Algebra   303 (2006), no. 1, 295--308.

 \bibitem[MP1]{MP} J.\,Mermin and I.\,Peeva, {\em Lexifying ideals}, Math. Res. Lett.13 (2006), no.2-3, 409--422.%, available at \url{https://math.okstate.edu/people/mermin/papers/lexifying_ideals.pdf}.

\bibitem[MP2]{MP2}J.\,Mermin and I.\,Peeva,
{\em Hilbert functions and lex ideals} J. Algebra   313 (2007), no. 2, 642--656.

\end{thebibliography}
\end{document}